\documentclass[11pt,a4paper]{article}
\usepackage[utf8]{inputenc}
\usepackage[english]{babel}
\usepackage[a4paper,margin=1in]{geometry}

\usepackage{amssymb}
\usepackage{amsfonts}
\usepackage{amsmath}
\usepackage{amsthm}

\usepackage{sectsty}
\sectionfont{\large}
\subsectionfont{\normalsize}
\subsubsectionfont{\small}

\usepackage{epsfig}
\usepackage{graphicx}
\usepackage{pgfplots}
\usepackage{lmodern}
\usepackage{dsfont}
\usepackage{mathrsfs}
\usepackage{latexsym}

\usepackage{verbatim}
\usepackage{setspace}
\usepackage{enumitem}
\usepackage{mathtools}

\renewcommand{\labelenumi}{\textup{(\roman{enumi})}}
\renewcommand\theenumi\labelenumi 

\theoremstyle{plain}
\newtheorem{theorem}{Theorem}[section]
\newtheorem{corollary}[theorem]{Corollary}
\newtheorem{lemma}[theorem]{Lemma}

\theoremstyle{definition}
\newtheorem{definition}[theorem]{Definition}
\newtheorem{remark}[theorem]{Remark}

\newtheorem{example}[theorem]{Example}

\newcommand{\test}{\cC^\infty_c(\real^d)}
\newcommand{\domain}{\mathscr{D}(\cA)}
\newcommand\diff{\mathrm{d}}
\newcommand{\real}{\mathds{R}}
\newcommand{\complex}{\mathds{C}}
\newcommand{\nat}{\mathds{N}}

\newcommand{\X}{(X_t)_{t\geq0}}

\newcommand{\testo}{\cC_c^\infty(\real)}

\newcommand{\hess}{\mathbf{H}}
\newcommand{\ccA}{\mathscr{A}}

\newcommand{\PP}{\mathds{P}}

\newcommand{\NN}{\mathbb{N}}

\newcommand{\EE}{\mathds{E}}
\newcommand{\cB}{\mathcal{B}}
\newcommand{\cC}{\mathcal{C}}

\newcommand{\cD}{\mathcal{D}}

\newcommand{\cA}{\mathcal{A}}

\newcommand{\cE}{\mathcal{E}}

\newcommand{\cM}{\mathcal{M}}

\newcommand{\ee}{\mathrm{e}}
\newcommand{\bone}{\mathds{1}}

\newcommand{\supp}{\operatorname{supp}}

\newcommand{\sgn}{\operatorname{sgn}}

\newcommand{\loc}{\mathrm{loc}}

\renewcommand{\epsilon}{\varepsilon} 

\usepackage{xcolor}

\definecolor{see}{HTML}{AA7700}
\definecolor{doubt}{HTML}{AF0000}
\definecolor{todo}{HTML}{0000AA}

\usepackage{marginnote}
\begin{document}
\frenchspacing
\allowdisplaybreaks[4]

\title{\bfseries Invariant measures of L\'evy-type operators and their associated Markov processes}
\author{Anita Behme\thanks{Technische Universit\"at
		Dresden, Institut f\"ur Mathematische Stochastik, Helmholtzstra{\ss}e 10, 01069 Dresden, Germany. \texttt{anita.behme@tu-dresden.de},  \texttt{david.oechsler@tu-dresden.de}},\;  David Oechsler$^\ast$}
\date{\today}
\maketitle
\begin{abstract}
A distributional equation as a criterion for invariant measures of Markov processes associated to Lévy-type operators is established. This is obtained via a  characterization of infinitesimally invariant measures of the associated generators. Particular focus is put on the one-dimensional case where the distributional equation becomes a Volterra-Fredholm integral equation, and on solutions to Lévy-driven stochastic differential equations. The results are accompanied by various illustrative examples. 


\end{abstract}

\section{Introduction}
\setcounter{equation}{0}

One of the most fundamental questions regarding stochastic systems is the description of their long-time behavior, i.e. the question whether the respective system admits one or more invariant distributions, and under which conditions these are obtained on the long run.\\

Given a Markov process \(\X\) with generator \((\cA,\domain)\) it is well-known that, under suitable regularity conditions, a measure \(\eta\) is invariant for \(\X\) if and only if 
\begin{align}\label{def_inv_law}
\int \cA f \diff \eta = 0
\end{align}
for all $f$ in some subset $D$ of the domain $\domain$. For example, this equivalence holds true if \(\eta\) is finite, \(\X\) is a Feller process, \((\cA,\domain)\) is its Feller generator, and \(D\) is a core of \((\cA,\domain)\), cf. \cite[Sec. 3]{liggett}. In general, however, \eqref{def_inv_law} is not equivalent to $\eta$ being an invariant measure of a Markov process, in which case solutions of \eqref{def_inv_law} are called infinitesimally invariant. More precisely, given a Banach space  \((\cE, \|\cdot\|)\) of functions \(f:\real^d\to\real\), and  a linear operator \((\cA,\domain)\)  on \(\cE\) such that the smooth functions with compact support $\test$ are contained in the domain $\domain$ of the generator, a measure \(\eta\) on \(\real^d\) is called \textbf{infinitesimally invariant} for \((\cA,\domain)\) if and only if \eqref{def_inv_law} holds for all \(f\in\test\). Note at this point that \eqref{def_inv_law} only requires that \(\cA f\in L^1(\eta)\) for all \(f\in\test\), and in particular (infinitesimal) invariant measures \(\eta\) may be infinite. \\

In this article we focus on Markov processes on \(\real^d\) whose generator is a Lévy-type operator. Heuristically speaking, these processes \textit{behave locally like Lévy processes}, and their class most notably includes rich Feller processes, i.e. Feller processes for which  $\test \subset \domain$. These in turn include Lévy processes, Feller diffusions, and solutions to certain stochastic differential equations (SDEs) such as e.g. Ornstein-Uhlenbeck-type processes (OUT) or generalized Ornstein-Uhlenbeck (GOU) processes. Nonetheless, the class we consider in this article also contains processes which are not necessarily rich Feller processes, e.g. many solutions to Lévy-driven SDEs. \\
Lévy-type operators and the associated Markov processes have been an active area of research for the last two to three decades. For a general introduction to this field see \cite{bottcher,jacob2001levy} and the references therein. Some more recent articles such as \cite{kuhn2018martingale, kuhn} are concerned with the questions, whether a given Lévy-type operator is the generator of a rich Feller process, and under which conditions the solution of a Lévy-driven SDE is a rich Feller process. General stability properties of processes associated to Lévy-type operators, such as non-explosion, recurrence, and ergodicity, are  discussed in \cite{sandric,wang2,wang}. In certain special cases such as e.g. GOU or OUT processes, integro-differential equations for invariant measures have been derived, see \cite{behme4,kuznetsov,sato}. In \cite{albeverio2,albeverio} (infinitesimally) invariant measures of solutions to certain classes of Lévy-driven SDEs are described. Other related sources such as \cite{bogachev} focus on the infinite-dimensional case, while \cite{stannat} considers inifinitesimally invariant measures for operators without jump component. Lastly we mention that in \cite{behme,behme3} invariant measures of Itô processes (of which rich Feller processes are a strict subclass) are described via the characteristic function.\\

In this article we aim to derive a general criterion for (infinitesimally) invariant measures of Lévy-type operators and their associated Markov processes. These criteria will be formulated as a distributional equation for the invariant measure and our methods are closely related to those used in \cite{behme4,kuznetsov,sato}. We start our article with the preliminaries where we recapitulate distribution theory as far as we need it for our purposes, and fix our notation.  In Section~\ref{sec_LToper} we then derive a new representation of a Lévy-type operator which we use in Section \ref{sec_infiinv} to prove a distributional equation for the infinitesimally invariant measures of a Lévy-type operator. Here we also highlight the one-dimensional case in which the distributional equation can be transformed into a  Volterra-Fredholm integral equation. Section \ref{sec_processes} then deals with the application of the prior results to invariant distributions of Markov processes associated to Lévy-type operators. Section~\ref{sec_sde} focuses on solutions to Lévy-driven SDEs and provides a necessary and sufficient condition for an invariant probability law of a solution of an SDE, if the solution process is a Feller process. Moreover, the applicability to more general SDEs is demonstrated. Finally, in Section \ref{sec_Adj} we discuss an implication of our results on the adjoint of a Lévy-type operator.

\section{Preliminaries}\label{sec_prelim}
\setcounter{equation}{0}

\subsection{Notations}

Throughout this article we write \(\cB_b(\real^d)\) for the Banach space of bounded Borel measurable functions \(f:\real^d\to\real\) equipped with the sup norm $\| \cdot \|_\infty$. Further we write \(\cC_b(\real^d)\subset\cB_b(\real^d)\) for the subspace of bounded continuous functions on \(\real^d\), $\cC_c^k(\real^d)\subset \cC_b(\real^d)$, $k\in\NN$, for the subspaces of $k$-times continuously differentiable functions with compact support, \(\cC_c^\infty(\real^d)\subset\cC_c^k(\real^d)\) for the subspace of test functions on \(\real^d\), that is the space of smooth functions with compact support, and \(\cC_\infty(\real^d)\subset\cC_b(\real^d)\) for the subspace of continuous functions vanishing at infinity, i.e., continuous functions \(f:\real^d\to\real\) such that for all \(\epsilon>0\) there exists a compact set \(K\subset \real^d\) such that for all \(x\in \real^d\setminus K\) it holds \(|f(x)|<\epsilon\).
For an open set \(\Omega\subset\real^d\) we denote \(\cB_b(\Omega)\), \(\cC_c^\infty(\Omega)\), and alike for the respective subspaces of functions with support in $\Omega$.\\
For any measure $\eta$ we write $L^p(\eta)$ and $W^{k,p}(\eta)$ for the classic Lebesgue and Sobolev spaces w.r.t. $\eta$.\\ 
Open and closed balls with radius \(r>0\) centered around \(x\in\real^d\) are denoted by \(B(x,r):=\{y\in\real^d: |y-x|<r\}\) and \(\overline B(x,r):=\{y\in\real^d: |y-x|\leq r\}\), respectively.\\
Partial derivatives of a function $f:\real^d\to\real$  will be written as \(\partial_i f\), \(i=1,\ldots,d\), while the gradient operator is denoted by \(\nabla :=(\partial_i)^T_{i=1,\ldots,d}\), and the Hessian operator by \(\hess:=(\partial_i\partial_j)_{i,j=1,\ldots,d}\).  Directional derivatives along a given vector \(y\in\real^d\setminus\{0\}\) are denoted by \(\partial_y:=y\cdot\nabla\). \\ 
Moreover, given two vectors $y,z\in \real^d$, $y\cdot z=\sum_{i=1}^d y_i z_i$ is their scalar product. Likewise, given a vector-valued function \(f:\real^d\to\real^d, x\mapsto (f_i(x))_{i=1,\ldots,d}\) we write \(\nabla\cdot f:= \sum_{i=1}^d\partial_i f_i\), and similarly, \(\hess\cdot M:=\sum_{i,j=1}^d \partial_i\partial_jM_{ij}\) for a matrix-valued function \(M:\real^d\to\real^{d\times d},x\mapsto (M_{ij}(x))_{i,j=1,\ldots,d}\).

\subsection{Distribution theory} 

Let $\Omega\subseteq \real^d$ be open, and abbreviate \(\cD(\Omega):=\cC_c^\infty(\Omega)\). Then we write \(\cD'(\Omega)\) for the space of all \textbf{distributions} \(T:\cD(\Omega) \to \real, f\mapsto T(f)=\langle T, f\rangle \), i.e. the space of all linear functionals that are continuous w.r.t. uniform convergence on compact subsets of all derivatives. If \(\Omega=\real^d\), which will be the case most of the time, we write \(\cD':=\cD'(\real^d)\). \\
As it is well known, \(T\in\cD'(\Omega)\) if for all compact sets \(K\subset \Omega\) there exist a constant \(C_K>0\) and a number \(N_K\in\nat\) such that for all functions \(f\in \cD(\Omega)\) with \(\supp f\subset K\) it holds
\begin{align}\label{def_distr}
	|\langle T, f\rangle|\leq C_K\max\{\|\partial^\alpha f\|_\infty: |\alpha|\leq N_K\}.
\end{align}
This allows to define the \textbf{order} \(N_T\in\nat\cup\{\infty\}\) of \(T\) as 
\begin{align*}
	N_T:=\inf\{n\in\nat: \eqref{def_distr} \text{ holds for all } K\subset \Omega \text{ compact with } N_K=n\}.
\end{align*} 
If $T\in \cD'(\Omega)$ is of order $k\in \NN$, it can be uniquely extended to a linear functional on \(\cD^k(\Omega):=\cC^k_c(\Omega)\) which is continuous w.r.t. uniform convergence on compact subsets of all derivatives up to order \(k\). We write \(\cD'^k(\Omega)\) for the space of these extensions. \\
Distributions of order $0$ can be represented by Radon measures: A measure \(m\) on \((\Omega,\cB(\Omega))\) which is inner regular and locally finite, i.e. for which 
\begin{align*}
	m(\Omega')= \sup\{m(K): K\subset\Omega' \text{ compact}\} \quad \text{for any open set }\Omega'\subset\Omega,
\end{align*}
and such that for all \(x\in\Omega\) there exists an open set \(\Omega''\) such that \(x\in\Omega''\) and \(m(\Omega'')<\infty\), is called a \textbf{Radon measure} on \((\Omega,\cB(\Omega))\). The difference \(m:=m_1-m_2\) of two Radon measures \(m_1,m_2\) is called a \textbf{real-valued Radon measure} on \((\Omega,\cB(\Omega))\), and we write \(|m|=m_1+m_2\). \\
By the Riesz representation theorem, for every distribution \(T\in\cD'^0(\Omega)\) there exists a real-valued Radon measure \(m\) on \((\Omega,\cB(\Omega))\) such that
\begin{align}\label{eq_radon}
	\langle T, f\rangle = \int_{\Omega} f(x)m(\diff x)
\end{align}
for all \(f\in \cD^0(\Omega)\). Vice versa, every real-valued Radon measure on \((\Omega,\cB(\Omega))\) defines a distribution in \(\cD'^0(\Omega)\) via the mapping \eqref{eq_radon}.\\
Further, a distribution \(T\in\cD'(\Omega)\) is called \textbf{regular} if there exists a function \(g\in L^1_\loc(\real^d)\) such that
\begin{align*}
	\langle T, f\rangle =\int_{\real^d} f(x)g(x)\diff x
\end{align*} 
for all \(f\in\cD(\Omega)\). From the above it is clear that all regular distributions are of order $0$.\\
We refer to \cite{brezis2011functional, duistermaat} for references of the above and further background information.\\ 

Since we can identify 
\begin{itemize}
\item a distribution of order \(k\in\nat\) with its unique continuous extension in \(\cD'^k\),
\item a distribution of order \(0\) with its associated real-valued Radon measure, and
\item a regular distribution with its associated locally integrable function,
\end{itemize}
we will sometimes abuse notation and use the same symbol for the two respective objects. The correct interpretation will always be clear from the given context. \\

We will frequently deal with the question whether the multiplication of a function \(g:\real^d\to\real\) and a distribution \(\eta\in\cD'^0\), i.e. the mapping \(g\eta: f\mapsto \int_{\real^d} f(x) g(x)\eta(\diff x)\), remains a distribution. The following Lemma provides conditions for this.  Its elementary proof is included as we were unable to find a suitable reference.
\begin{lemma}\label{lem_abcM}
	Let \(\eta\) be a real-valued Radon measure on \(\real^d\).
	\begin{enumerate}[resume]
		\item If \(g\in L^1_\loc(\eta)\), then \(g\eta\in\cD'^0\).
		\item If \(g\in W^{1,1}_\loc(\eta)\), i.e. if all first weak derivatives of $g$ exist and are in $L^1_\loc(\eta)$, then \(g\partial_i\eta\in\cD'^1\) for all \(i\in\{1,\ldots,d\}\). 
	\end{enumerate}
\end{lemma}
\begin{proof}
	(i) Let \(f\in\cC_c(\real^d)=\cD^0\) and let \(K\subset\real^d\) be compact such that \(\supp f\subseteq K\). By assumption, \(\left|\int_{\real^d} f(x)g(x)\eta(\diff x)\right|\leq \|f\|_\infty\left|\int_{K} g(x)\eta(\diff x)\right|<\infty\), which implies the claim. \\
(ii) Assume \(g\in W^{1,1}_\loc(\eta)\). From (i) it follows that \(g\eta \in\cD'^0\) and \((\partial_i g) \eta\in \cD'^0\) for all \(i\in\{1,\ldots,d\}\). Moreover, the product rule \(\partial_i(g\eta) = (\partial_i g)\eta + g\partial_i \eta\) applies. Indeed, let \(f\in\cC_c^\infty(\real^d)\), then we have
	 \begin{align*}
	 \langle \partial_i (g\eta), f\rangle &= -\int_{\real^d} \partial_i f(x) g(x) \eta(\diff x)\\
	 &= \int_{\real^d} \left(f(x)\partial_ig(x)-\partial_i(f(x)g(x))\right)\eta(\diff x)\\
	 &=  \langle (\partial_ig) \eta,f\rangle + \langle g\partial_i \eta, f\rangle,
	 \end{align*}
	 because the product rule holds for \(f\) and \(g\), see \cite[Sec. 6.15]{rudinFA}. Since \(\partial_i(g\eta) \in \cD'^1\) and \((\partial_i g)\eta\in\cD'^0\) assertion (ii) follows.
\end{proof}

\section{Lévy-type operators}\label{sec_LToper}
\setcounter{equation}{0}

Let \(\cE\) be some Banach space of functions \(f: \real^d\to\real\). A \textbf{Lévy-type operator} is a pair \((\cA,\domain)\) consisting of a subset \(\domain\subset \cE\) and an integro-differential operator \(\cA: \domain \to \cE\) of the form 
\begin{align}\label{eq-operator}
	\cA f(x) = ~~&a(x)\cdot \nabla f(x) +  \frac12 \nabla \cdot b(x)\nabla f(x)\nonumber \\
	&+ \int_{\real^d}(f(x+y)-f(x)-\nabla f(x)\cdot y\bone_{\{|y|<1\}})\Pi(x,\diff y), \quad f\in\domain,
\end{align}
where \(a\) is an \(\real^d\)-valued Borel measurable function, \(b\) is a symmetric positive semidefinite \(\real^{d\times d}\)-valued Borel measurable function, and \(\Pi\) is a Lévy kernel in \(\real^d\), i.e. \(\Pi(x,\diff z)\) is a \(d\)-dimensional Lévy measure for all fixed  \(x\in\real^d\), and \(\Pi(\cdot,B)\) is Borel measurable for all \(B\subset\cB(\real^d)\). For notational convenience we assume throughout that  $\Pi(x,\{0\})=0$ for all \(x\in\real^d\). We call \((a,b,\Pi)\) the \(\mathbf{x}\)\textbf{-dependent Lévy triplet} or \textbf{characteristic triplet} of \((\cA,\domain)\). Whenever needed, we address the components of \(a\) and \(b\) by writing \(a(x)=(a(x)_i)_{i=1,\ldots,d}\) and \(b(x)=(b(x)_{ij})_{i,j=1,\ldots,d}\).\\

To characterise the (infinitesimally) invariant measures associated to the Lévy-type operator \((\cA,\domain)\), we intend to interpret \eqref{def_inv_law} in a distributional sense. To this aim, in this section we derive a slightly different representation of \(\cA\) that allows to isolate the test functions $f\in \cE$.\\

Our derivation relies on some preparatory concepts, and we start with a decomposition of the Lévy kernel as it is often used in the special case of Lévy processes, i.e. of Markov processes generated by Lévy-type operators for which the three components of the characteristic triplet \((a,b,\Pi)\) are constant in \(x\in\real^d\). However, other than in the context of Lévy processes, where a suitable cut-off function typically separates the jump measure into two measures, we will go one step further and separate the Lévy kernel into three different kernels.

\begin{definition}\label{def_decompose}
	Let \(\Pi\) be a Lévy kernel in \(\real^d\) and let \(h_\nu,h_\mu,h_\rho\in\cC(\real^d)\) be three non-negative, continuous functions such that 
	$$\Pi(x,\diff y)=h_\nu(y)\Pi(x,\diff y)+h_\mu(y)\Pi(x,\diff y)+h_\rho(y)\Pi(x,\diff y)\;\text{ for all }\;x\in\real^d.$$ 
	Setting \(\nu:=h_\nu\Pi,\mu:=h_\mu\Pi\), and \(\rho:=h_\rho\Pi\), we call the decomposition \(\pi=(\nu,\mu,\rho)\) \textbf{convenient}, if 
	\begin{enumerate}
		\item \(\int_{\real^d}\nu(x,\diff y)<\infty\) for all \(x\in\real^d\),
		\item \(\int_{\real^d} (1\wedge |y|) \mu(x,\diff y) <\infty\) for all \(x\in\real^d\),
		\item \(\int_{\real^d} (|y|\wedge |y|^2) \rho(x,\diff y) <\infty\) for all \(x\in\real^d\).
	\end{enumerate}
	For any convenient decomposition, the components \(\nu,\mu\), and \(\rho\) are referred to as the \textbf{large}, \textbf{medium}, and \textbf{small jump kernel} of \(\Pi\), respectively.
\end{definition}

To avoid confusion we will throughout this text use the Greek letters \(\nu,\mu,\rho,\) and $\pi$ exclusively in their respective roles in the previous definition. 

\begin{remark}\label{rem_decomp}
	Note that for any Lévy kernel \(\Pi\) in \(\real^d\) there exists a convenient decomposition. For example set \(h_\nu(y)=\frac{|y|^2}{|y|^2+|y|+1},~h_\mu(y)=\frac{|y|}{|y|^2+|y|+1}\) and \(h_\rho(y)=\frac{1}{|y|^2+|y|+1}\). \\
	Moreover, one can always find a convenient decomposition such that at least one of the three components vanishes. For example, the medium jump kernel can be discarded by setting \(h_\nu(y):=\frac{|y|^2}{|y|^2+1}\), \(h_\mu(y)=0\) and \(h_\rho(y):=\frac{1}{|y|^2+1}\). However, it will sometimes be preferable to work with, say, the medium jump kernel instead of the small jump kernel or vice versa, if both options are available. 
\end{remark}

Besides the above decomposition w.r.t. the jump sizes, we also need to establish a polar decomposition of the jump kernel. To start with this, note that, by \cite[Lem. 59.3]{sato2nd}, for any \(\sigma\)-finite measure \(\Pi\) on \(\real^d\) satisfying \(\Pi(\{0\})=0\), there exist a finite measure \(\Pi_{o}\) on \(S^{d-1}:=\{\xi\in\real^d; |\xi|=1\}\) with \(\Pi_{o}(S^{d-1})\geq0\) and a family \((\Pi_\xi)_{\xi\in S^{d-1}}\) of \(\sigma\)-finite measures on \((0,\infty)\) with \(\Pi_\xi((0,\infty))>0\) such that \(\Pi_\xi(B)\) is measurable in \(\xi\) for each \(B\in\cB((0,\infty))\) and 
\begin{align*}
	\Pi(B)= \int_{S^{d-1}}\int_{(0,\infty)}\bone_B(r\xi) \Pi_\xi(\diff r)\Pi_{o}(\diff \xi).
\end{align*}
Any such pair \((\Pi_o,(\Pi_\xi)_{\xi\in S^{d-1}})\) is called a \textbf{polar decomposition} of \(\Pi\) with \(\Pi_{o}\) being referred to as \textbf{spherical part} and \((\Pi_\xi)_{\xi\in S^{d-1}}\) as \textbf{radial part} of \(\Pi\). Again by \cite[Lem. 59.3]{sato2nd}, the pair \((\Pi_o,(\Pi_\xi)_{\xi\in S^{d-1}})\) is unique up to multiplication with a measurable function, i.e. for any other polar decomposition \((\hat\Pi_o,(\hat\Pi_\xi)_{\xi\in S^{d-1}})\) there exists a measurable function \(c:S^{d-1}\to (0,\infty)\) such that \(c(\xi)\hat\Pi_o(\diff \xi)=\Pi_o(\diff \xi)\), and \(\hat\Pi_\xi(\diff r)=c(\xi)\Pi_\xi(\diff r)\) for \(\Pi_o\)-a.e. \(\xi\in S^{d-1}\).\\

The following Lemma implies in particular, that for any Lévy measure we can even find a polar decomposition, such that the radial part is again a collection of Lévy measures:
\begin{lemma}\label{cor_levydec}
	Let \(\Pi\) be a \(\sigma\)-finite measure on
	\(\real^d\) with \(\Pi(\{0\})=0\). Let
	\(f:(0,\infty)\to\real_+\) be a function such
	that \(\int_{\real^d} f(|x|)\Pi(\diff x)<\infty\).
	Then there exists a polar decomposition
	\((\Pi_o,(\Pi_\xi)_{\xi\in S^{d-1}})\) of \(\Pi\) such that
	\(\int_{(0,\infty)} f(r)\Pi_\xi(\diff r)<\infty\)
	for all \(\xi\in S^{d-1}\). 
\end{lemma}
\begin{proof}
	Clearly, there exists a polar decomposition \((\Pi_o,(\Pi_\xi)_{\xi\in S^{d-1}})\) of \(\Pi\). By virtue of the assumption on  \(\Pi\) it holds
	\begin{align*}
		\int_{S^{d-1}}\int_{(0,\infty)} f(r) \Pi_\xi(\diff r)\,\Pi_{o}(\diff \xi)<\infty.
	\end{align*}
	Thus, 	\(I:=\{\xi\in S^{d-1}; \int_{(0,\infty)}f(r)\Pi_\xi(\diff r) =\infty\}\) 
	is a \(\Pi_o\)-nullset, and in particular measurable. We may therefore set \(\Pi_\xi=0\) on \(I\) to obtain the required polar decomposition.
\end{proof}

The factorization into polar coordinates can now be applied pointwise to a Lévy kernel in \(\real^d\). This enables the following definition.

\begin{definition} \label{def_inttail}
	Let \(\Pi\) be a Lévy kernel in \(\real^d\) and for all fixed \(x\in\real^d\) choose a polar decomposition \((\Pi_{o}(x,\cdot),(\Pi_\xi(x,\cdot))_{\xi\in S^{d-1}})\)  of \(\Pi(x,\cdot)\) as in Lemma \ref{cor_levydec}. Then we define \(\overline\Pi\), the \textbf{integrated radial tail} of \(\Pi\), by  
	\begin{align*}
		\overline\Pi(x,B):=\int_{S^{d-1}}\int_{(0,\infty)} \bone_{B}(z\xi) \int_{(z,\infty)} \Pi_\xi(x,\diff r)\, \diff z\, \Pi_{o}(x,\diff \xi), \quad B\in\cB(\real^d).
	\end{align*}
Moreover, if \(\int_{\real^d}(|y|\wedge |y|^2) \Pi(x,\diff y)<\infty\) for all \(x\in\real^d\), we set for all \(B\in\cB(\real^d)\)
	\begin{align*}
		\overline{\overline\Pi}(x,B):= \int_{S^{d-1}}\int_{(0,\infty)} \bone_{B}(z_2\xi) \int_{(z_2,\infty)}\int_{(z_1,\infty)} \Pi_\xi(x,\diff r)\,\diff z_1\, \diff z_2\, \Pi_{o}(x,\diff \xi),
	\end{align*}
and call \(\overline{\overline\Pi}\) the \textbf{double-integrated radial tail} of \(\Pi\).
\end{definition}

The following Lemma proves well-definedness of the (double-)integrated radial tail.

\begin{lemma}
Let \(\Pi\) be a Lévy kernel in \(\real^d\). Then the integrated radial tail \(\overline\Pi\) of \(\Pi\) is independent of the choice of the polar decomposition, \(\overline\Pi(x,\diff z)\) is a \(d\)-dimensional \(\sigma\)-finite measure with \(\overline\Pi(x,\{0\})=0\) for all \(x\in\real^d\), and \(\overline\Pi(\cdot,B)\) is Borel measurable for all \(B\in\cB(\real^d)\). \\ The same holds true for the double-integrated radial tail if \(\int_{\real^d}(|y|\wedge |y|^2) \Pi(x,\diff y)<\infty\).
\end{lemma}

\begin{proof}
	Fix \(x\in\real^d\), then clearly \(\overline\Pi(x,\{0\})=0\), and for all \(\xi\in S^{d-1}\) the radial part \(\Pi_\xi(x,\cdot)\) is a one-dimensional Lévy measure. Therefore, \(z\mapsto\int_{(z,\infty)} \Pi_\xi(x,\diff r)\) is a locally bounded function on \((0,\infty)\). Hence, \(\overline\Pi_\xi(x,B):=\int_B\int_{(z,\infty)} \Pi_\xi(x,\diff r)\,\diff z\) is a \(\sigma\)-finite measure for fixed \(\xi\in S^{d-1}\) and, as $\Pi_o$ is a finite measure, $\overline\Pi(x, \diff z)$ is also a $\sigma$-finite measure.\\
	For any fixed \(B\in\cB(\real^d)\) by Tonelli's theorem we derive
	\begin{align*}
		\overline\Pi(x,B)
		&=\int_{S^{d-1}}\int_{(0,\infty)} \int_{(0,r)}\bone_{B}(z\xi) \diff z \, \Pi_\xi(x,\diff r)\, \Pi_{o}(x,\diff \xi)\\
		&=\int_{\real^d} \int_{(0,1]}\bone_{B}\big(z y \big) \diff z \, \Pi(x,\diff y)
	\end{align*}
	which is independent of the choice of the polar decomposition.\\ Further, approximating \(f:y\mapsto \int_{(0,1]}\bone_{B}(z  y) \diff z\) by simple functions, the measurability of \(\overline\Pi(\cdot,B)\) now follows from the measurability of \(\Pi(\cdot,B)\) and \(f\), and the monotone convergence theorem.\\
	For the double-integrated radial tail simply note that, due to \(\int_{\real^d}(|y|\wedge |y|^2) \Pi(x,\diff y)<\infty\) for all \(x\in\real^d\), we obtain
	\begin{align*}
	\int_{(z_2,\infty)}  \int_{(z_1,\infty)} \Pi_\xi(x,\diff r)\diff z_1= \int_{(z_2,\infty)} \int_{(0,r)} \diff z_1 \Pi_\xi(x,\diff r) =\int_{(z_2,\infty)} r \Pi_\xi(x,\diff r)<\infty 
	\end{align*}
	for all \(z_2>0\) and \(\Pi_o(x,\cdot)\)-a.e. \(\xi\in S^{d-1}\). The rest follows analogously.
\end{proof}

We are now in the position to formulate the desired representation of the operator \(\cA\).

\begin{lemma} \label{lem_genrep}
Let \((\cA,\domain)\) be a Lévy-type operator with characteristic triplet \((a, b, \Pi)\) and convenient decomposition \(\pi=(\nu,\mu,\rho)\) of the Lévy kernel \(\Pi\). Then for any \(f\in\cC^2_c(\real^d)\cap\domain\) it holds
\begin{align*}
\cA f(x)= &~ (a(x)+a_\pi(x))\cdot \nabla f(x)+ \frac12 \nabla \cdot b(x)\nabla f(x) \\
&+ \int_{\real^d} f(x+y)\nu(x,\diff y) - f(x)\nu(x,\real^d) \\
&+ \int_{\real^d} \partial_{y/|y|} f(x+y) {\overline\mu}(x,\diff y) 
+ \int_{\real^d} \partial^2_{y/|y|} f(x+y) \overline{\overline\rho}(x,\diff y),
\end{align*}
where \(a_\pi(x)= -\int_{|y|<1}y\nu(x,\diff y)- \int_{|y|<1}y\mu(x,\diff y)+\int_{|y|\geq1}y\rho(x,\diff y)\).
\end{lemma}

\begin{proof}
Inserting the convenient decomposition of \(\Pi\) on \eqref{eq-operator} yields
\begin{align*}
\cA f(x)&=a(x)\cdot \nabla f(x) + \frac12 \nabla \cdot b(x)\nabla f(x) \\
&\quad +\int_{\real^d}\left(f(x+y)-f(x)\right)\nu(x,\diff y) - \int_{|y|<1}y\nu(x,\diff y) \cdot \nabla f(x)\\
&\quad +\int_{\real^d}\left(f(x+y)-f(x)\right)\mu(x,\diff y) - \int_{|y|<1}y\mu(x,\diff y)\cdot \nabla f(x) \\
&\quad +\int_{\real^d}\left(f(x+y)-f(x)-\nabla f(x)\cdot y\right)\rho(x,\diff y) +\int_{|y|\geq1}y\rho(x,\diff y)\cdot \nabla f(x) \\
&:= (a(x)+a_\pi(x))\cdot \nabla f(x) + \frac12 \nabla \cdot b(x)\nabla f(x)  + I_\nu + I_\mu + I_\rho,
\end{align*}
where the integral terms of \(a_\pi\) are finite for all \(x\in\real^d\) due to the assumptions on \(\nu,\mu\) and \(\rho\). \\
First, since \(\nu(x,\real^d)<\infty\) for all \(x\in\real^d\), we immediately obtain
\begin{align*}
I_\nu=\int_{\real^d} f(x+y)\nu(x,\diff y) - f(x)\nu(x,\real^d).
\end{align*} 
Next, consider \(I_\mu\). The integrand's argument \(y\in \real^d\setminus\{0\}\) can be uniquely written as \(y=r\xi\) for some \(r>0\) and \(\xi\in S^{d-1}\). Using a polar decomposition \((\mu_o,(\mu_\xi)_{\xi\in S^{d-1}})\) of \(\mu\), and the fact that \(f(x+r\xi)-f(x)=\int_{(0,r)}\partial_\xi f(x+z\xi)\diff z\), this yields
\begin{align}\nonumber
I_\mu&=\int_{S^{d-1}}\int_{(0,\infty)}(f(x+r\xi)-f(x))\mu_\xi(x,\diff r)\,\mu_{o}(x,\diff \xi)\\\nonumber
&=\int_{S^{d-1}}\int_{(0,\infty)}\int_{(0,r)} \partial_\xi f(x+z\xi) \diff z \,\mu_\xi(x,\diff r)\,\mu_{o}(x,\diff \xi)\\\label{eq_polarc}
&=\int_{S^{d-1}}\int_{(0,\infty)}\partial_\xi f(x+z\xi) \int_{(z,\infty)}  \mu_\xi(x,\diff r)\,\diff z\,\mu_{o}(x,\diff \xi),
\end{align}
where we applied Fubini's theorem to interchange the two inner integrals.
This is possible since \(\partial_\xi f\) is bounded and with compact support, i.e. for fixed \(x\in\real^d\) and \(\xi\in S^{d-1}\) there are constants \(C,C'>0\) such that \(|\partial_\xi f(x+z\xi)|<C' \bone_{\{|z|<C\}}\). Therefore, to show \(I_\mu<\infty\), it suffices to verify that the area of integration
\begin{align*}
B_1:=\left\{(r,z)\in\real^2: 0< z< r <C\right\}
\end{align*}
has finite mass w.r.t. the product measure \(\mu_\xi (x,\cdot) \times \lambda\), where \(\lambda\) denotes the Lebesgue measure. This however follows due to the fact that \(\int_{\real^d} |y| \mu(x,\diff y)<\infty\) which implies by Lemma \ref{cor_levydec} that $\int_{(0,\infty)} r \mu_\xi(x,\diff r)<\infty$ and hence 
$$\int_{B_1} \big(\mu_\xi (x,\diff r) \times \diff z\big) = \int_{(0,C)} r \mu_\xi (x,\diff r)<\infty.$$ 
Lastly, converting \(z\xi\) back into Cartesian coordinates in \eqref{eq_polarc} yields the desired result, namely
\begin{align*}
I_\mu = \int_{y\neq0} \partial_{y/|y|} f(x+y) {\overline\mu}(x,\diff y).
\end{align*}
It remains to rewrite \(I_\rho\). Again, writing \(y=r\xi\) with \(r>0\) and \(\xi\in S^{d-1}\), we observe that
\begin{align*}
f(x+y)-f(x)-\nabla f(x)\cdot y&= f(x+r\xi)-f(x)-\nabla f(x)\cdot (r\xi) \\
&= \int_{(0,r)} (\partial_\xi f(x+z_1\xi) - \partial_\xi f(x) ) \diff z_1 \\
&=\int_{(0,r)}\int_{(0,z_1)} \partial^2_\xi f(x+z_2\xi) \diff z_2 \,\diff z_1.
\end{align*}
We obtain
\begin{align*}
I_\rho&=\int_{S^{d-1}}\int_{(0,\infty)}(f(x+r\xi)-f(x)-\nabla f(x)\cdot (r\xi))\rho_\xi(x,\diff r)\,\rho_{o}(x,\diff \xi)\\
&=\int_{S^{d-1}}\int_{(0,\infty)}\int_{(0,r)}\int_{(0,z_1)} \partial^2_\xi f(x+z_2\xi) \diff z_2 \,\diff z_1\,\rho_\xi(x,\diff r)\,\rho_{o}(x,\diff \xi)\\
&=\int_{S^{d-1}}\int_{(0,\infty)}\partial^2_\xi f(x+z_2\xi) \int_{(z_2,\infty)}\int_{(z_1,\infty)} \rho_\xi(x,\diff r) \,\diff z_1\, \diff z_2\, \rho_{o}(x,\diff \xi),
\end{align*}
where we again used Fubini's theorem, the applicability of which is shown by a similar argumentation as above:  The area of integration
\[
B_2:=\{(r,z_1,z_2)\in\real^3: 0< z_2< z_1< r< C\}
\]
is finite w.r.t the product measure \(\rho_\xi(x,\cdot)\times \lambda\times\lambda\) since \(\int_{(0,\infty)} (r \wedge r^2)\rho_\xi(x,\diff r)<\infty\) for all \(x\in\real^d\) due to Lemma \ref{cor_levydec}. 
By the assumption on \(\rho\), the double-integrated radial tail \(\overline{\overline\rho}\) of \(\rho\) is well-defined. Converting \(z_2\xi\) back into Cartesian coordinates thereby yields
\begin{align*}
I_\rho=\int_{y\neq0} \partial^2_{y/|y|} f(x+y) \overline{\overline\rho}(x,\diff y),
\end{align*}
and the proof is complete.
\end{proof}

In the following, if we use the notation \(a_\pi\), we will always mean the function defined in Lemma~\ref{lem_genrep}.

\section{Infinitesimally invariant measures of Lévy-type operators}\label{sec_infiinv}
\setcounter{equation}{0}

\subsection{A distributional equation for infinitesimally invariant measures}

  Let \((\cA,\domain)\) be a Lévy-type operator  with characteristic triplet \((a,b,\Pi)\)  as defined in \eqref{eq-operator}, and assume \(\test\subset\domain\). Given a convenient decomposition \(\pi=(\nu,\mu,\rho)\) of \(\Pi\) it    then    follows from Lemma \ref{lem_genrep} and the definition of infinitesimal invariance that a measure \(\eta\) is infinitesimally invariant for \((\cA,\domain)\) if and only if
\begin{align}\label{eq_genequ}
\begin{split}
0=&~\int_{\real^d} \Big((a(x)+a_\pi(x))\cdot \nabla f(x) +  \frac12 \nabla \cdot b(x)\nabla f(x)\Big)\eta(\diff x)  \\
&+\int_{\real^d}\left( \int_{\real^d} f(x+y)\nu(x,\diff y)-f(x)\nu(x,\real^d)\right) \eta(\diff x)\\
&+\int_{\real^d} \int_{\real^d} \partial_{y/|y|} f(x+y) {\overline\mu}(x,\diff y)\, \eta(\diff x)\\
&+\int_{\real^d}\int_{\real^d} \partial^2_{y/|y|} f(x+y) \overline{\overline\rho}(x,\diff y)\, \eta(\diff x)
\end{split}
\end{align}
for all \(f\in\test\subset \cC^2_c(\real^d)\cap\domain\). \\
From this we derive a distributional equation of \(\eta\) in Theorem \ref{thm_distEQ} below. In order to state our result in a compact way we first introduce the following notation.

\begin{definition}\label{def_jumpfunc}
Let \((\cA,\domain)\) be a Lévy-type operator in \(\real^d\), and let   \(\pi=(\nu,\mu,\rho)\)    be a convenient decomposition of \(\Pi\). Let \(\eta\) be a measure on \(\real^d\). We define the \textbf{jump functionals} \(U^\pi(\eta),V^\pi(\eta)=(V_i^\pi(\eta))_{i=1,\ldots,d}\), and \(W^\pi(\eta)=(W_{ij}^\pi(\eta))_{i,j=1,\ldots, d}\) of \(\Pi\) w.r.t.   \(\pi\) and  \(\eta\), by setting 
\begin{align*}
\langle U^\pi(\eta),f \rangle &:=\int_{\real^d}\left( \int_{\real^d} f(x+y)\nu(x,\diff y)-f(x)\nu(x,\real^d)\right) \eta(\diff x),\\
\langle V_{i}^\pi(\eta),f\rangle&:=\int_{\real^d}\int_{\real^d} f(x+y) \frac{y_i}{|y|}\,\overline\mu(x,\diff y)\, \eta(\diff x), \quad i=1,\ldots, d,\\
\langle W_{ij}^\pi(\eta),f\rangle &:=\int_{\real^d} \int_{\real^d}f(x+y)\frac{y_iy_j}{|y|^2}\, \overline{\overline\rho}(x,\diff y)\, \eta(\diff x), \quad i,j=1,\ldots, d,
\end{align*}
for all \(f\in\test\) for which the respective right-hand side is well-defined. \\  
 Moreover, we set 
 \begin{align*}
 	\cM(\cA,\pi)&:=\big\{\text{measures }\eta \text{ on }\real^d \text{ s.t. } \nabla \cdot ((a+a_\pi)\eta), {\frac12 \nabla \cdot b\nabla\eta}, U^\pi(\eta), {\nabla\cdot V^\pi(\eta)},\\
 	& \qquad  \text{ and  } \hess \cdot W^\pi(\eta) \text{ define distributions} \big\}.\end{align*}
\end{definition}

We are now ready to state the main result of this section.

\begin{theorem}\label{thm_distEQ}
Let \((\cA,\domain)\) be a Lévy-type operator in \(\real^d\) with characteristic triplet \((a,b,\Pi)\), and assume that \(\test\subset\domain\). Let \(\pi=(\nu,\mu,\rho)\)    be a convenient decomposition of \(\Pi\). A measure \(\eta\in \cM(\cA,\pi)\) is infinitesimally invariant for \((\cA,\domain)\) if and only if \(\eta\) solves the distributional equation
\begin{align}\label{eq_distributional}
- \nabla \cdot ((a+a_\pi)\eta)+ \frac12 \nabla \cdot b\nabla\eta  + U^\pi(\eta) +\nabla \cdot V^\pi(\eta) + \hess \cdot W^\pi(\eta) =0,
\end{align}
where \(U^\pi(\eta), V^\pi(\eta)\) and \(W^\pi(\eta)\) denote the jump functionals of \(\Pi\) w.r.t.   \(\pi\) and    \(\eta\).
\end{theorem}

Note that for measures $\eta \notin \cM(\cA,\pi)$ we cannot interpret \eqref{eq_distributional} in a distributional sense. Consequently, the criterion fails for infinitesimally invariant measures \(\eta\notin\cM(\cA,\pi)\). However, in many cases mild regularity assumptions on the characteristic triplet \((a,b,\Pi)\) already guarantee that \(\cM(\cA,\pi)\) contains a sufficiently large class of measures. 
General conditions for a measure $\eta$ to be contained in $\cM(\cA,\pi)$ will be discussed in Section \ref{sec_cond} below.

\begin{proof}[Proof of Theorem \ref{thm_distEQ}]
  
We consider \eqref{eq_genequ} and transform the summands one-by-one.    First of all, abbreviating \(\tilde{a}(x):=a(x)+a_\pi(x) \) we observe
\begin{align}\nonumber
 \int_{\real^d} \tilde a(x) \cdot \nabla f(x)\, \eta(\diff x)&=\sum_{i=1}^d \int_{\real^d} \tilde a_i(x) \partial_i f(x) \,\eta(\diff x)\\
 &= - \sum_{i=1}^d \int_{\real^d} f(x) \partial_i \big(\tilde{a}_i(x) \,\eta(\diff x) \big)
\label{eq_a}
=-\langle \nabla \cdot(\tilde a\eta), f\rangle,
\end{align}
where the latter is by assumption a distribution. Likewise  \begin{align}\nonumber
\int_{\real^d} \frac12 \nabla \cdot b(x)\nabla f(x) \eta(\diff x)&= \sum_{i,j=1}^d \int_{\real^d}  \frac12 \partial_i \big( b_{ij}(x) \partial_j f(x)\big) \,\eta(\diff x)\\
\label{eq_b}
&= \frac12 \langle \nabla \cdot b^T\nabla\eta , f\rangle=\frac12 \langle \nabla \cdot b\nabla\eta , f\rangle,
\end{align}
and again this is a distribution.\\
Since \(U^\pi(\eta)\) appears in both \eqref{eq_genequ} and \eqref{eq_distributional} in the same form and is by assumption a distribution, it needs no further discussion. \\ 
For the term in the third line in \eqref{eq_genequ} we note that by definition $\partial_{y/|y|} f(x+y)  = \sum_{i=1}^d \partial_i f(x+y)\frac{y_i}{|y|}$ and hence 
\begin{align}\nonumber
\int_{\real^d} \int_{\real^d} \partial_{y/|y|} f(x+y)\, {\overline\mu}(x,\diff y) \,\eta(\diff x) & = \sum_{i=1}^d \int_{\real^d} \int_{\real^d} \partial_i f(x+y) \frac{y_i}{|y|}\, {\overline\mu}(x,\diff y)\, \eta(\diff x) \\
&= \sum_{i=1}^d \langle V_i^\pi(\eta), \partial_i f \rangle \label{eq_V}
= \langle \nabla \cdot V^\pi(\eta), f\rangle.
\end{align}
Here, the second line is well-defined as \(\nabla \cdot V^\pi(\eta)\) is a distribution by assumption. \\
To consider the final term, we first observe that
\begin{align}
	\partial^2_{y/|y|} f(x+y)&=\sum_{j=1}^d\sum_{i=1}^d \partial_j\Big(\partial_i f(x+y) \frac{y_i}{|y|}\Big)\frac{y_j}{|y|} \nonumber \\
	&= \sum_{j=1}^d\sum_{i=1}^d \Big(\partial_j\partial_i f(x+y) \Big) \frac{y_iy_j}{|y|^2} + \sum_{j=1}^d\sum_{i=1}^d \partial_i f(x+y) \Big(\partial_j \frac{y_i}{|y|}\Big)\frac{y_j}{|y|} \label{eq_2nddirdiff}
\end{align}
where
$$\partial_j\frac{y_i}{|y|}= \begin{cases}
	\frac{\sum_{\ell\neq i} y_\ell^2}{|y|^3}, & i=j,\\
	-\frac{y_i y_j}{|y|^3}, & i\neq j,
\end{cases}$$
which implies that the second term in \eqref{eq_2nddirdiff} is $0$. Thus 
\begin{align}\nonumber
\int_{\real^d}\int_{\real^d} \partial^2_{y/|y|} f(x+y)\, \overline{\overline\rho}(x,\diff y)\, \eta(\diff x) &= \sum_{i,j=1}^d\int_{\real^d} \int_{\real^d}\partial_j\partial_i f(x+y)\frac{y_iy_j}{|y|^2}\,\overline{\overline\rho}(x,\diff y)\, \eta(\diff x) \\ 
&= \sum_{i,j=1}^d \langle W^\pi_{ij}(\eta), \partial_i\partial_j f\rangle \label{eq_W}
= \langle \hess \cdot W^\pi(\eta),f\rangle.
\end{align}
Again this is well-defined as $\hess \cdot W^\pi(\eta)$ is a distribution by assumption. \\
   Summing up \eqref{eq_a}, \eqref{eq_b}, \eqref{eq_V}, \eqref{eq_W} and \(U(\eta)\) thus yields the equivalence of \eqref{eq_genequ} and \eqref{eq_distributional} and finishes the proof.   
\end{proof}

\subsection{Volterra-Fredholm integral equations}\label{sec_onedim}

In this section we will embed the distributional equation \eqref{eq_distributional} into the theory of Volterra-Fredholm integral equations: 
Under the restriction that $d=1$, in Theorem \ref{thm_vfie} we present integral kernels \(\kappa_{1,2}: \real\times \real\to\real\) and functions \(F_{1,2}:\real\to \real\), such that \eqref{eq_distributional} is equivalent to the \textbf{Volterra-Fredholm integral equation}
\begin{align}\label{eq_vfie_template}
F_1(x)\eta(\diff x) = \left( F_2(x) + \int_{\real} \kappa_1(x,z) \eta(\diff z) + \int_{(0,x]} \kappa_2(x,z)\eta(\diff z)\right)\diff x,
\end{align}
where for $x\leq 0$ the integral $\int_{(0,x]}\kappa_2(x,z)\eta(\diff z)$ is interpreted as $-\int_{(x,0]}\kappa_2(x,z)\eta(\diff z)$. \\
Thus, roughly speaking, a measure \(\eta\) on \(\real\) is infinitesimally invariant for a given Lévy-type operator, if and only if it solves the corresponding Volterra-Fredholm integral equation. 
For a detailed introduction and overview of such (and more general) equations see e.g. \cite{brunner2017volterra,gripenberg1990volterra,wazwaz2011linear}.\\

 Although theory on multidimensional integral equations exists, cf. \cite{bart1973linear,brunner2017volterra}, and certain special cases such as OUT processes, cf. \cite{sato}, even yield such integral equations, it is not possible to simply translate the steps of the proof of Theorem \ref{thm_vfie} to the multidimensional case. Thus, in order to avoid an overload of technicalities, we stick to the one-dimensional case. \\
 In the following, for any Lévy kernel $\Pi(x,\cdot)$, $x\in\real$, in $\real$ we denote 
 \begin{align}\label{eq_1d_tail}
 	\widetilde\Pi(x,z):=\begin{cases}
 		\int_{(z,\infty)}\Pi(x,\diff r) = \Pi(x,(z,\infty)), &z>0,\\
 		\int_{(-\infty,z)}\Pi(x,\diff r) = \Pi(x,(-\infty,z)), &z<0,\\
 		0, &  z=0.
 	\end{cases}
 \end{align}
 
 \begin{theorem}\label{thm_vfie}
 	Let \((\cA,\domain)\) be a Lévy-type operator in \(\real\) with characteristic triplet \((a,b,\Pi)\) and assume \(\testo\subset\domain\). Further, let \(\pi=(\nu,\mu, \rho)\) be a convenient decomposition of \(\Pi\) and let \(\eta\) be a positive Radon measure on \(\real\)  such that $a+a_\pi\in L^1_\loc(\eta)$, \(b\in W^{1,1}_\loc(\eta)\), and $U^\pi(\eta)$, $V^\pi(\eta)$, $W^\pi(\eta)\in\cD'^0$.\\
 	Then \(\eta\) is infinitesimally invariant for \((\cA,\domain)\) if and only if there exist constants \(c_1,c_2\in\real\) such that \(\eta\) solves \eqref{eq_vfie_template} with
 	\begin{align*}
 		F_1(x)& =\frac12b(x),\qquad  F_2(x)=c_1x+c_2,\\
 		 \kappa_1(x,z)& := \begin{cases} - \int_{(0,x]}\int_{(0,u]} \nu(z,\diff y-z) \diff u - \int_{(0,x]} \sgn(y-z) \widetilde\mu(z,y-z)\diff y - \widetilde{\overline\rho}(z,x-z), & x>0,\\
 		 - \int_{(x,0]}\int_{(u,0]} \nu(z,\diff y-z) \diff u + \int_{(x,0]} \sgn(y-z) \widetilde\mu(z,y-z)\diff y - \widetilde{\overline\rho}(z,x-z),	&x\leq 0, \end{cases} \\
 		\kappa_2(x,z)& :=  a(z) +  a_\pi(z) + \frac12b'(z) + (x-z)\nu(z,\real),
 	\end{align*}
 	for all \(x,z\in\real\), and where the densities \(\widetilde\mu\) and \(\widetilde{\overline\rho}\) are defined according to \eqref{eq_1d_tail}.
 \end{theorem}

 For the proof of this theorem we need two additional lemmas, the first of which shows that the (double) integrated radial tails of a Lévy kernel as defined in \ref{def_inttail} have an explicit representation in the one-dimensional case.

\begin{lemma}\label{lem_onedimkernel}
Let \(\Pi\) be a Lévy kernel in \(\real\). Then for all fixed $x\in \real$ the measure \(\overline\Pi(x,\cdot)\) is absolutely continuous and 
\begin{align}\label{eq_onedimrad}
\overline\Pi(x,B)&=\int_B \Pi(x,\{u\in\real: u\sgn(z)> z\sgn(z), z\neq0 \})\diff z  = \int_B \widetilde{\Pi}(x,z) \diff z,
\end{align}
for all \(B\in\cB(\real)\). Moreover, if \(\overline{\overline\Pi}(x,\cdot)\) is well-defined, then it is absolutely continuous as well and 
\begin{align}\label{eq_onedimint}
\overline{\overline\Pi}(x,B)=\int_B \overline\Pi(x,\{u\in\real: u\sgn(z)> z\sgn(z),z\neq 0\})\diff z = \int_B \widetilde{\overline{\Pi}}(x,z) \diff z,
\end{align}
for all \(x\in\real\) and \(B\in\cB(\real)\).
\end{lemma}
\begin{proof}
Since \(S^0=\{-1,1\}\) is discrete we obtain
\begin{align*}
\overline\Pi(x,B)&= \int_{\{-1,1\}}\int_{(0,\infty)} \bone_B(z\xi) \int_{(z,\infty)} \Pi_{\xi}(x,\diff r)\, \diff z \, \Pi_{o}(x,\diff \xi)\\
&=\int_{(0,\infty)} \bone_B(z) \int_{(z,\infty)} \Pi(x,\diff r)\,\diff z +\int_{(-,\infty,0)} \bone_B(z) \int_{(-\infty,z)} \Pi(x,\diff r)\,\diff z,
\end{align*}
for all \(x\in\real\) and \(B\in\cB(\real)\). This already implies that for fixed \(x\in\real\) the measure $\overline{\Pi}(x,\cdot)$  is absolutely continuous with density $\widetilde{\Pi}(x,\cdot)$ 
in agreement with \eqref{eq_onedimrad}.\\
 Further, if \(\overline{\overline\Pi}\) is well-defined, we have by definition
\begin{align*}
\overline{\overline\Pi}(x,B) & = \int_{\{-1,1\}}\int_{(0,\infty)} \bone_{B}(z_2\xi) \int_{(z_2,\infty)}\int_{(z_1,\infty)} \Pi_\xi(x,\diff r)\, \diff z_1 \,\diff z_2 \,\Pi_o(x,\diff \xi)\\
&= ~~\int_{(0,\infty)} \bone_{B}(z_2) \int_{(z_2,\infty)}\int_{(z_1,\infty)} \Pi(x,\diff r)\,\diff z_1\, \diff z_2 \\
&~~+ \int_{(-\infty,0)} \bone_{B}(z_2) \int_{(-\infty,z_2)}\int_{(-\infty,z_1)} \Pi(x,\diff r)\,\diff z_1 \,\diff z_2
\end{align*}
for all \(x\in\real\) and \(B\in\cB(\real)\). Thus, the density of the double-integrated radial tail \(\overline{\overline\Pi}\) can be written as
\begin{align*}
\widetilde{\overline\Pi}(x,z) = \begin{cases}
\int_{(z,\infty)} \Pi(x,(y,\infty))\diff y  = \overline{\Pi}(x, (z,\infty)), &z>0,\\
\int_{(-\infty,z)} \Pi(x,(-\infty,y))\diff y  = \overline{\Pi}(x, (-\infty, z)), &z<0,\\
 0, &  z=0,
\end{cases}
\end{align*}
in agreement with \eqref{eq_onedimint}.
\end{proof}

\begin{lemma}\label{lem_primitiveradon}
	Let $T\in \cD'^0(\real)$ and let $m$ be the associated real-valued Radon measure as in \eqref{eq_radon}. Then the function 
	$$M:\real\to\real:x\mapsto  \int_{(0,x]} m(\diff y) := \begin{cases}
		m((0,x])=:\langle T, \mathds{1}_{(0,x]}\rangle , & x> 0,\\
		-m((x,0])=:-\langle T, \mathds{1}_{(x,0]}\rangle,& x<0,\\ 
		0,& x=0,	\end{cases}$$ 
		is locally integrable. Moreover, for every $f\in \cC_c^\infty(\real)$ it holds
$$\langle T, f \rangle = - \int_\real f'(x) M(x) \diff x = - \langle M, f' \rangle,$$
i.e. $M$ defines a distributional primitive of $T$.
\end{lemma}
\begin{proof}
	By definition $m$ is the difference of two (positive) Radon measures and we can assume w.l.o.g. that $m>0$. Then $M$ is non-decreasing and as $M(x)\in \real$ for all $x\in \real$ this implies the local integrability. \\
	Further, for $f\in \cC_c^\infty(\real)$ we compute by Fubini's theorem
	\begin{align*}
		\langle M, f'\rangle &= \int_\real f'(x) M(x) \diff x \\
		&= \int_{(0,\infty)} f'(x)  \int_{(0,x]}  m(\diff u) \, \diff x - \int_{(-\infty,0)} f'(x)  \int_{(x,0]}  m(\diff u) \, \diff x\\
		&= \int_{(0,\infty)} \int_{[u,\infty)} f'(x)   \diff x\,  m(\diff u) -  \int_{(-\infty,0]} \int_{(-\infty,u)} f'(x)  \diff x \,  m(\diff u)\\
		&= - \int_{(0,\infty)} f(u) m(\diff u) -  \int_{(-\infty,0]}  f(u)  m(\diff u) = - \langle T, f \rangle. \qedhere
	\end{align*}
\end{proof}

\begin{proof}[Proof of Theorem \ref{thm_vfie}] 
First of all note that by Lemma \ref{lem_abcM} (see also Corollary \ref{cor_etainM} below) our assumptions imply that \(\eta\in\cM(\cA,\pi)\).\\	
As \(d=1\), Equation \eqref{eq_distributional} simplifies to 
\begin{align}\label{eq_dist1dim}
-((a+a_\pi)\eta)' + \frac12 (b\eta')' + U^\pi(\eta) + V^\pi(\eta)' + W^\pi(\eta)''=0,
\end{align}
where we emphasize that differentiation and integration is to be understood in the distributional sense. By assumption, \(\frac12 (b\eta')'\) defines a distribution, and therefore it admits a primitive \(\frac12 (b\eta')\) which again defines a distribution. Moreover, as $\eta \in \cD'^0$ and $b\in W^{1,1}_\loc(\eta)$  the product rule \(b\eta'=(b\eta)'-b'\eta\) applies as seen in the proof of Lemma \ref{lem_abcM}. In particular, \(b\eta\) and \(b'\eta\) define distributions as well. Hence, integrating \eqref{eq_dist1dim} twice yields
\begin{align}\label{eq_int1dim}
\frac12b\eta-\int\Big(a+a_\pi + \frac12 b'\Big)\eta + \iint U^\pi(\eta) +\int V^\pi(\eta) + W^\pi(\eta)=F_2, 
\end{align}
where \(F_2\) denotes the distribution associated to the locally integrable function \(F_2(x)=c_1x+c_2\).  \\
Since primitives of distributions are unique only up to a constant, leaving \(c_1,c_2\in\real\) unspecified allows us to freely choose the primitives on the left-hand side of \eqref{eq_int1dim}. We can therefore  associate the first integral term  to the locally integrable function $$x\mapsto\begin{cases}
	-\int_{ (0,x]}(a(y)+a_\pi(y)+\frac12b'(y))\eta(\diff y), & x>0,\\
	  \int_{(x,0]}(a(y)+a_\pi(y)+\frac12b'(y))\eta(\diff y), & x<0,\\
	0,& x=0.
\end{cases} $$  
With the notational convention introduced in the beginning of this section this mapping simply reads as
$$ x\mapsto
	-\int_{(0,x]}(a(y)+a_\pi(y)+\frac12b'(y))\eta(\diff y). $$  
To derive the second-order primitive of \(U^\pi(\eta)\), first use Lemma \ref{lem_primitiveradon} to observe that we obtain a locally integrable function associated to a distributional primitive of \(U^\pi(\eta)\) by mapping 
\begin{align*}
u\mapsto \int_\real \int_{(0,u]} \nu(z,\diff y-z)\eta(\diff z) - \int_{(0,u]} \nu(z,\real)\eta(\diff z).
\end{align*} 
Integrating once again and swapping the order of integration by Tonelli's theorem we obtain 
\begin{align*}
\int_{(0,x]} \int_\real \int_{ (0,u]} \nu(z,\diff y-z)\eta(\diff z)  \diff u
&=\int_\real\int_{(0,x]} \int_{(0,u] }\nu(z,\diff y-z)\,\diff u \, \eta(\diff z).
\end{align*}
Likewise 
\begin{align*}
\int_{(0,x]} \int_{(0,u]} \nu(z,\real)\eta(\diff z) \diff u 
&= \int_{(0,x]} \int_{[z,x]} \nu(z,\real) \diff u \, \eta(\diff z) 
 =\int_{(0,x]}(x-z)\nu(z,\real)\, \eta(\diff z).
\end{align*}
Thus, the locally integrable function
\begin{align*}
	x\mapsto \int_\real\int_{(0,x]} \int_{(0,u] \setminus\{z\} }\nu(z,\diff y-z)\,\diff u \, \eta(\diff z) - \int_{ (0,x]}(x-z)\nu(z,\real)\, \eta(\diff z)
\end{align*}
is associated to a second order primitive of $U^\pi(\eta)$.\\
In the same spirit , applying Lemma \ref{lem_onedimkernel} on the Lévy kernel $\overline{\mu}$,
\begin{align*}
x &\mapsto  \int_\real \int_{(0,x]} \sgn (y-z) \overline{\mu}(z, \diff y - z)\, \eta(\diff z)
= \int_{\real} \int_{(0,x]} \sgn(y-z)\widetilde\mu(z, y-z)\diff y\, \eta(\diff z),
\end{align*}
is a locally integrable function associated to a primitive of $V^\pi(\eta)$.\\
 Lastly, we compute, again via Lemma \ref{lem_onedimkernel}, 
\begin{align*}
\langle W^\pi(\eta), f\rangle &=\int_\real\int_\real f(x)\overline{\overline\rho}(z,x-z)\diff x\,  \eta(\diff z)
=\int_\real f(x)\int_\real \widetilde{\overline\rho}(z,x-z)\eta(\diff z)\,\diff x
\end{align*} 
for all \(f\in\testo\), due to the Fubini-Tonelli theorem. \\
We have thereby shown that all terms of \eqref{eq_int1dim} are distributions of order \(0\), and thus, written in terms of the associated real-valued Radon measures, \eqref{eq_int1dim} becomes \eqref{eq_vfie_template} with \(F_1,F_2,\kappa_1,\kappa_2\) as given in the statement of the theorem.
\end{proof}

\begin{remark}
	In many cases one can use additional assumptions on the invariant measure (e.g. it being a probability law) to specify one or both of the constants \(c_1\) and \(c_2\) in Theorem \ref{thm_vfie}; see e.g. \cite{behme4} or the upcoming example in Section \ref{sec_superlinear}. However, there seems to be no general method to deal with them.
\end{remark}

The following corollary addresses absolute continuity of solutions to \eqref{eq_vfie_template}.

\begin{corollary}\label{cor_abscont}
Let \(\eta\) be an infinitesimally invariant measure under the assumptions of Theorem \ref{thm_vfie}. Set \(\Omega_1:=\{x\in\real: b(x)>0\}\) and  let \(\Omega_2\subset\real\) be an open set such that
\begin{enumerate}
\item \(a(x)+a_\pi(x)\neq 0\) for all \(x\in \Omega_2\), and 
\item \(\langle W^\pi(\eta),f\rangle=0\) for all \(f\in \cC_c^\infty(\Omega_2)\).
\end{enumerate}
Then there exist an absolutely continuous measure \(\eta_c\), and a measure \(\eta_d\) with \(\supp(\eta_d)\subset\real\setminus~(\Omega_1\cup\Omega_2)\) such that \(\eta=\eta_c+\eta_d\).
\end{corollary}
\begin{proof}
We prove that $\eta$ is absolutely continuous on $\Omega_1$ and $\Omega_2$. Setting  \(\eta_c:= \eta\big|_{\Omega_1\cup\Omega_2}\) then implies the statement.\\
First, note that as $F_1(x)=\frac12 b(x)>0 $ for all $x\in \Omega_1$, we derive from \eqref{eq_vfie_template} for all \(f\in\cC_c^\infty(\Omega_1)\) that 
\begin{align*}
	\int_\real f(x)\eta(\diff x) = \int_\real f(x) \frac{2}{b(x)} \left( F_2(x) + \int_{\real} \kappa_1(x,z) \eta(\diff z) + \int_{(0,x]} \kappa_2(x,z)\eta(\diff z)\right) \diff x,
\end{align*}
where the term in brackets on the right-hand is locally integrable. Thus $\eta$ is absolutely continuous on $\Omega_1$.\\
Second, by assumption, \(U^\pi(\eta)\) and \(V^\pi(\eta)\) have associated real-valued Radon measures on \(\Omega_2\). Denoting the associated measure of $V^\pi(\eta)$ by $V$, and integrating \eqref{eq_dist1dim} once, on \(\Omega_2\) we obtain 
\begin{align*}
(a(x)+a_\pi(x)) \eta(\diff x) = \left(\langle U^\pi(\eta), \mathds{1}_{(0,x]} \rangle + c\right)\diff x + V(\diff x), \qquad x>0,
\end{align*}
for some constant \(c\in\real\), and a similar equality holds for $x\leq 0$.  Thus,
\begin{align*}
\int_\real f(x)\eta(\diff x) = \int_\real f(x) \frac{\langle U^\pi(\eta), \mathds{1}_{(0,x]} \rangle + c}{a(x)+a_\pi(x)} \diff x+ \int_\real f(x) \frac{1}{a(x)+a_\pi(x)} V(\diff x)
\end{align*}
for all \(f\in\cC_c^\infty(\Omega_2)\). By definition and Lemma \ref{lem_onedimkernel} we have 
\begin{align*}
	\langle V^\pi(\eta), f\rangle &= \int_\real \int_{y>0} f(x+y) \tilde{\mu}(x,y) \diff y \, \eta(\diff x) - \int_\real \int_{y<0} f(x+y) \tilde{\mu}(x,y) \diff y \, \eta(\diff x)\\
	&= \int_{y>0} \int_\real f(x+y) \tilde{\mu}(x,y)  \eta(\diff x)\, \diff y - \int_{y<0} \int_\real f(x+y) \tilde{\mu}(x,y) \eta(\diff x)\, \diff y \\
	&= \int_\real f(z) \int_\real \sgn(z-x) \tilde{\mu}(x,z-x) \eta(\diff x) \, \diff z,
\end{align*}
where in the penultimate step we used the Fubini-Tonelli theorem for the integrals w.r.t. the positive and negative part of $V$. Thus, $V$ is absolutely continuous which finishes the proof.
\end{proof}

\subsection{On the space \(\cM(\cA,\pi)\)}\label{sec_cond}

Although we have not been able to find an example of a Lévy-type operator \((\cA,\domain)\) with an infinitesimally invariant measure \(\eta\) such that \(\eta\notin\cM(\cA,\pi)\), the assumption that \(\eta\in\cM(\cA,\pi)\) in Theorem \ref{thm_distEQ} is not automatically fulfilled for all measures one might wish to consider. For instance, let \(\eta(\diff x)=\bone_{\{x>0\}} x^{-1/2}\diff x\) and \((a+a_\pi)(x)=\bone_{\{x>0\}} x^{-1/2}\). Then \(((a+a_\pi)\eta)'\) is not a distribution.\\

Nonetheless, from Lemma \ref{lem_abcM} we immediately obtain the following simple conditions under which the terms \(\nabla\cdot((a+a_\pi)\eta)\) and \(\frac12\nabla\cdot b\nabla\eta\) define distributions. Recall that we write $g\eta \in  \cD'^k$ if $g\eta$ defines a distribution via \eqref{eq_radon}, i.e. if $f\mapsto  \int f(x) g(x) \eta(\diff x)$
is a distribution in $\cD'^k$. 

\begin{corollary}\label{cor_etainM}
	Let \(\eta\) be a real-valued Radon measure on \(\real^d\).
	\begin{enumerate}[resume]
		\item If \(a+a_\pi\in L^1_\loc(\eta)\) then \(\nabla \cdot ((a+a_\pi)\eta\in \cD'^1\).
		\item If \(b\in W^{1,1}_\loc(\eta)\), then \(\frac12 \nabla \cdot b \nabla \eta\in \cD'^2\). 
	\end{enumerate}
\end{corollary}

Unfortunately, the jump functionals $U^\pi(\eta)$, $V^\pi(\eta)$ and $W^\pi(\eta)$ of a Lévy kernel \(\Pi\) w.r.t. a convenient decomposition \(\pi=(\nu,\mu,\rho)\) and a measure \(\eta\) are not necessarily distributions either. This is illustrated by the following counterexample.

\begin{example}\label{ex_counter}
	Let \(d=1\) and consider the kernel \(\Pi(x,\diff y)=\frac{|x|^{1+\alpha+\beta}\diff y}{1\vee|y|^{1+\alpha}}\) where \(\alpha,\beta\in(0,1)\).  Then $\int_{y\neq 0} \Pi(x, \diff y)<\infty$ and we can choose \(\pi=(\Pi,0,0)\) as convenient decomposition. \\
	Fix some real numbers \(a<b\). Then we obtain for all \(x>(1+b)\vee 0\)
	\begin{align*}
		\int_{[a-x,b-x]}\Pi(x,\diff y) & =  \frac{x^{1+\alpha+\beta}}\alpha ((x-b)^{-\alpha}-(x-a)^{-\alpha}) =\frac{x^{\beta}}\alpha \cdot \frac{\left(\frac{x}{x-b}\right)^\alpha - \left(\frac{x}{x-a}\right)^\alpha}{x^{-1}}.
	\end{align*} 
	L'Hospital's rule shows
	\begin{align*}
		\lim_{x\to\infty}\frac{\left(\frac{x}{x-b}\right)^\alpha - \left(\frac{x}{x-a}\right)^\alpha}{x^{-1}} &= \lim_{x\to\infty}\frac{\alpha\left(\frac{x}{x-b}\right)^{\alpha-1}\frac{b}{(x-b)^2} - \left(\frac{x}{x-a}\right)^{\alpha-1}\frac{a}{(x-a)^2}}{x^{-2}}\\
		&=\lim_{x\to\infty}\alpha\left(\frac{x}{x-b}\right)^{\alpha-1}\frac{bx^2}{(x-b)^2} - \alpha \left(\frac{x}{x-a}\right)^{\alpha-1}\frac{ax^2}{(x-a)^2}\\
		&=\alpha (b-a),
	\end{align*}
	which implies that \(\int_{[a-x,b-x]}\Pi(x,\diff y) \sim (b-a)x^\beta\) as  \(x\to \infty\). \\
	Therefore, choosing any probability measure \(\eta\) for which \(x\mapsto \bone_{\{x>1\}} |x|^\beta\) is not integrable, we observe that  
	\begin{align*}
		\langle U^\pi(\eta),\bone_{[a,b]}\rangle= \int_\real\Big(\int_{[a-x,b-x]}\Pi(x,\diff y) - \bone_{[a,b]}(x)\Pi(x,\real)\Big)\eta(\diff x)
	\end{align*}
	does not converge. This implies that \(U^\pi(\eta): \cB(\real^d)\to\real\) is not locally finite and hence \(U^\pi(\eta)\) can neither be a real-valued Radon measure, nor a distribution of order \(0\).\\ 
	Moreover, since \(a<b\) in the above consideration were chosen arbitrarily, we may even conclude that for all test function \(f\in \cC_c^\infty(\real)\) with \(f\geq0\) and \(\supp (f) \neq \emptyset\) it holds \(U^\pi(\eta)(f)=\infty\). Thus, \(U^\pi(\eta)\) cannot be a distribution of any order for the chosen measure $\eta$ and the chosen decomposition of $\Pi$. \\
	Still, the question remains whether another convenient decomposition can be found such that the jump functionals w.r.t. the new decomposition are distributions. However, if the passing to another decomposition is possible, we are essentially dealing with the same objects. \\
To see this, observe that for the given Lévy kernel another suitable choice would be given by $\pi'=(0,\Pi,0)$. However, via Lemma \ref{lem_onedimkernel} one observes that for this choice \(\nabla\cdot V^{\pi'}(\eta)=U^{\pi}(\eta)\) for any real-valued Radon measure for which either side is defined. Thus also $\nabla\cdot V^{\pi'}(\eta)$ cannot be a distribution of any order for  $\eta$ as above. 		
\end{example}

Example \ref{ex_counter} implies that - apart from the conditions stated in Corollary \ref{cor_etainM} - criteria for $\eta\in~\cM(\cA,\pi)$ are needed. We therefore provide sufficient conditions to guarantee that $U^\pi(\eta)$, ${\nabla\cdot V^\pi(\eta)}$, and \(\hess \cdot W^\pi(\eta)\) are distributions in Lemma \ref{lem_D1} and the subsequent corollaries.\\ 
Recall that for any real-valued Radon measure \(\eta\) there exist two positive Radon measures \(\eta_+,\eta_-\) such that \(\eta=\eta_+-\eta_-\), and that \(|\eta|:=\eta_+ +\eta_-\).

\begin{lemma}\label{lem_D1}
	Let \(\Pi\) be a \(d\)-dimensional Lévy kernel with convenient decomposition \(\pi=(\nu,\mu,\rho)\) and let \(\eta\) be a real-valued Radon measure on \(\real^d\). Denote the jump functionals of \(\Pi\) w.r.t. \(\pi\) and \(\eta\) by \(U^\pi(\eta),V^\pi(\eta),W^\pi(\eta)\).
	\begin{enumerate}
		\item Assume that 
		\begin{align}\label{eq_cond1}
			\int_{\real^d} \left(\nu(x,B(-x,r))+\bone_{\{|x|<r\}}\nu(x,\real^d)\right)|\eta|(\diff x) <\infty
		\end{align}
		for all $r>0$. Then \(U^\pi(\eta)\in\cD'^0\).
		\item Assume that 
		\begin{align}\label{eq_cond2}
			\int_{\real^d} \overline{\mu}(x,B(-x,r)) \eta(\diff x)<\infty
		\end{align}
		for all \(r>0\). Then \(V^\pi_{i}\in\cD'^0\) for every \(i\in\{1,\ldots,d\}\).
		\item Assume that
		\begin{align}\label{eq_cond3}
			\int_{\real^d} \overline{\overline\rho}(x,B(-x,r)) \eta(\diff x)<\infty
		\end{align}
		for all $r>0$. Then \(W^\pi_{ij}\in\cD'^0\) for every \(i,j\in\{1,\ldots,d\}\).
	\end{enumerate}
\end{lemma}
\begin{proof}
Fix \(r>0\) and consider \(f\in \cC_c(B(0,r))\), then
	\begin{align*}
		|\langle U^\pi(\eta), f\rangle |&\leq \int_{\real^d} \left(\int_{\real^d}|f (z)| \nu(x,\diff z-x) + |f(x)|\nu(x,\real^d)\right)|\eta|(\diff x)\\
		&\leq \|f\|_\infty\int_{\real^d} \left(\nu(x,B(-x,r)) + \bone_{B(0,r)} (x)  \nu(x,\real^d)\right)|\eta|(\diff x).
	\end{align*}   
	By the characterization \eqref{def_distr} this implies that \(U^\pi(\eta)\) is a distribution of order \(0\).\\
		Assertions (ii) and (iii) are shown similarly using that  \(\left|\frac{y_i}{|y|}\right|\leq1\) and \(\left|\frac{(y_i)(y_j)}{|y|^2}\right|\leq1\) for all \(y\in\real^d\).
\end{proof}

Conditions \eqref{eq_cond2} and \eqref{eq_cond3} are formulated in terms of the integrated radial tail \(\overline\mu\) and the double-integrated radial tail \(\overline{\overline\rho}\), respectively. The following lemma presents two similar conditions that are stated in terms of \(\mu\) and \(\rho\) directly.

\begin{lemma}\label{cor_jumpII}
	Let \(\Pi\) be a \(d\)-dimensional Lévy kernel with convenient decomposition \(\pi=(\nu,\mu,\rho)\) and let \(\eta\) be a real-valued Radon measure on \(\real^d\). Denote the jump functionals of \(\Pi\) w.r.t. \(\pi\) and \(\eta\) by \(U^\pi(\eta),V^\pi(\eta),W^\pi(\eta)\).
	\begin{enumerate}
		\item If \(\int_{\real^d} \int_{\real^d} (1\wedge |y|)\mu(x,\diff y)\, \eta(\diff x)<\infty\), then \(V_i(\eta)\in\cD'^0\) for every \(i\in\{1,\ldots,d\}\).
		\item  If \(\int_{\real^d} \int_{\real^d} (|y|\wedge |y|^2)\rho(x,\diff y)\, \eta(\diff x)<\infty\), then \(W_{ij}(\eta)\in\cD'^0\) for all \(i,j\in\{1,\ldots,d\}\).
	\end{enumerate} 
\end{lemma}
\begin{proof}
	For this proof we use the notational convention \(\int_a^b:=\int_{(a,b)}\) to avoid lengthy formulas. \\
	By the definition of \(\overline\mu\) and the fact that \(B(-x,r)\subset B(0,|x|+r)\setminus \overline B(0,(|x|-r)\vee 0)\) we obtain 
	\begin{align*}
		\overline\mu(x,B(-x,r))&= \int_{S^{d-1}}\int_{0}^{\infty} \bone_{B(-x,r)}(z_1 \xi ) \int_{z_1}^{\infty} \mu_\xi(x,\diff z_2)\, \diff z_1 \, \mu_o(x,\diff \xi)\\
		&\leq\int_{S^{d-1}}\int_{(|x|-r) \vee 0}^{|x|+r} \int_{z_1}^\infty \mu_\xi (x,\diff z_2)\, \diff z_1 \, \mu_o(x,\diff \xi).
	\end{align*}
	Swapping the order of the inner two integrals using Tonelli's theorem this implies 
	\begin{align*}
		\overline\mu(x,B(-x,r)) &\leq \int_{S^{d-1}}\int_{(|x|-r)\vee 0}^\infty \int_{(|x|-r)\vee 0}^{(|x|+r)\wedge z_2} \diff z_1\,  \mu_\xi(x,\diff z_2)\,  \mu_o(x,\diff \xi)\\
		&\leq \int_{S^{d-1}}\int_0^\infty (2r \wedge z_2) \mu_\xi (x,\diff z_2) \, \mu_o(x,\diff \xi )\\ 
		&\leq C\int_{\real^d} (1\wedge|y|) \mu(x,\diff y)
	\end{align*}
	for some \(C>0\). Thus, the condition in (i) implies \eqref{eq_cond2} which proves the claim.\\
	For (ii) we use a similar approach and compute 
	\begin{align*}
		\overline{\overline\rho}(x,B(-x,r))& =\int_{S^{d-1}}\int_{0}^\infty \bone_{B(-x,r)}(z_1 \xi)\int_{z_1}^\infty\int_{z_2}^\infty \rho_\xi (x,\diff z_3)\, \diff z_2\, \diff z_1\, \rho_{o}(\diff \xi)\\
		& \leq  \int_{S^{d-1}}\int_{(|x|-r) \vee 0}^{|x|+r} \int_{z_1}^\infty \int_{z_2}^\infty \rho_\xi(x,\diff z_3)\, \diff z_2\, \diff z_1\, \rho_{o}(\diff \xi)\\
 &=  \int_{S^{d-1}}\int_{(|x|-r) \vee 0}^\infty \int_{(|x|-r)\vee 0}^{z_3}\int_{(|x|-r)\vee 0}^{(|x|+r)\wedge z_2} \diff z_1\, \diff z_2\, \rho_\xi(x,\diff z_3)\, \rho_{o}(\diff \xi) \\
 & \leq  \int_{S^{d-1}}\int_{(|x|-r) \vee 0}^\infty \int_{0}^{z_3} (2r \wedge z_2) \diff z_2\, \rho_\xi(x,\diff z_3)\, \rho_{o}(\diff \xi) \\
		&\leq \int_{S^{d-1}}\int_{0}^\infty \left( 2rz_3 \wedge \frac{z_3^2}2  \right) \rho_\xi (x,\diff z_3)\, \rho_{o}(\diff \xi) \\
		&\leq C \int_{\real^d}(|y|\wedge|y|^2)\rho(x,\diff y)
	\end{align*}
	for some \(C>0\). Therefore, the condition in (ii) implies \eqref{eq_cond3} and the proof is complete.
\end{proof}

We end this section with two corollaries that cover special cases. First we consider Lévy kernels that are space-invariant.

\begin{corollary}\label{cor_constantPi}
	Let \(\Pi\) be a \(d\)-dimensional Lévy measure, i.e. a Lévy kernel independent of \(x\), and let \(\pi=(\nu,\mu,\rho)\) be a convenient decomposition of \(\Pi\). Then for any finite measure \(\eta\) on \(\real^d\) the jump functionals \(U^\pi(\eta),V^\pi_i(\eta),W^\pi_{ij}(\eta)\) are in \(\cD'^0\) for every \(i,j\in\{1,\ldots,d\}\). 
\end{corollary}
\begin{proof}
	Note that in the present situation 
	\begin{align*}
		\nu(B(-x,r)) + \bone_{\{|x|<r\}}\nu(\real^d) \leq 2\nu(\real^d)
	\end{align*}
	for all \(x\in\real^d\). Thus the integrand in \eqref{eq_cond1} is bounded, and hence the integral converges. This implies  $U^\pi(\eta)\in \cD'^0$. For the other two jump functionals we may apply Corollary \ref{cor_jumpII}    observing that both \(\int_{\real^d} (1\wedge|y|) \mu(\diff y)\) and \(\int_{\real^d}(|y|\wedge|y|^2)\rho(\diff y)\) are finite and independent of \(x\).
\end{proof}

Lastly we consider Lévy kernels of a form as it often appears in the context of SDEs; see Section~\ref{sec_sde} below for more details.

\begin{corollary}\label{cor_sdefinite}
	Let \(\phi:\real^d\to\real^{d\times n}\) be a function and let \(\Pi\) be a finite \(n\)-dimensional Lévy measure. For \(x\in\real^d\) and \(B\in\cB(\real^d)\) set
	 $$\Pi_\phi(x,B):=\begin{cases}\Pi(\{y:\in\real^d; \phi(x)y\in B\}), &\phi(x)\neq0\\ 0, & \phi(x)=0.\end{cases}$$ 
	Then \(\pi=(\Pi_\phi,0,0)\) is a convenient decomposition of the Lévy kernel \(\Pi_\phi\).
  If \(\eta\) is a finite measure on \(\real^d\), then the jump functional \(U^\pi(\eta)\) of \(\Pi_\phi\) w.r.t. \(\pi\) and \(\eta\) is in \(\cD'^0\).
\end{corollary}
\begin{proof}
	From \(\Pi\) being finite it follows that \(\Pi_\phi(x,\cdot)\) is finite for all \(x\in\real^d\). Thus, \(\pi\) is a convenient decomposition of \(\Pi_\phi\). 
	Moreover, \(\Pi_\phi(x,B(-x,r))\leq \Pi(\real^n)<\infty\) for all $r>0$ which	implies that \(x\mapsto \Pi_\phi(x,B(-x,r))+\bone_{\{|x|<r\}}\Pi_\phi(\real^d)\) is bounded. Thus, Lemma \ref{lem_D1} (i) applies which proves the claim.
\end{proof}

\section{Invariant measures of Markov processes associated to Lévy-type operators}\label{sec_processes}
\setcounter{equation}{0}

This section deals with the application of our results to Markov processes associated to Lévy-type operators. As mentioned in the introduction, this class of processes contains many important cases such as Feller processes with sufficiently rich domain, or solutions of SDEs driven by Lévy processes.\\

Throughout this section, let \((X_t)_{t\geq0}\) be a time-homogeneous Markov process in \(\real^d\) with transition kernels \((p_t)_{t\geq0}\) given by
\begin{align*}
p_t(x,B):=\PP^{ x}(X_t\in B)=\EE^x\left[\bone_B(X_t)\right] := \EE[\bone_B(X_t)|X_0=x], \quad t\geq0, x\in\real^d, B\in \cB(\real^d),
\end{align*}
and associated transition semigroup 
\begin{align*}
	P_tf(x):=\EE^x\left[f(X_t)\right], \quad t\geq 0, x\in\real^d, f\in \cC_\infty(\real^d). 
\end{align*}
It is well-known that \((P_t)_{t\geq0}\) forms a contractive semigroup, i.e. \(\|P_t f\|_\infty \leq \|f\|_\infty \) for all \(f\in\cC_\infty(\real^d)\) and \(t\geq0\), and that the Markov property implies \(P_sP_t=P_{s+t}\) for all \(0\leq s\leq t\). 
The \textbf{pointwise (infinitesimal) generator} of \((X_t)_{t\geq0}\) is the pair \((\cA,\domain)\) defined by 
\begin{align*}
\cA f(x)&:=\lim_{t\downarrow 0}\frac{P_tf(x)-f(x)}{t}, \quad x\in\real^d, f\in\domain,
\end{align*} 
where
\begin{align}\label{eq_infdomain}
\domain&:=\left\{f\in \cC_\infty(\real^d): \lim_{t\downarrow 0}\frac{P_tf(x)-f(x)}{t} \text{ exists for all } x\in\real^d \right \}.
\end{align}

Assuming additionally that  \((X_t)_{t\geq0}\) is a \textbf{Feller process}, i.e. that the associated transition semigroup has the Feller properties
\begin{enumerate}
\item \(P_t f\in \cC_\infty(\real^d)\) for all \(t\geq0, f\in \cC_\infty(\real^d)\), and
\item \(\lim_{t\downarrow 0} \|P_tf-f\|_\infty=0\) for all \(f\in \cC_\infty(\real^d)\),
\end{enumerate}
the pair \((\ccA,\mathscr{D}(\ccA))\) defined by
\begin{align*}
\ccA f(x)&:=\lim_{t\downarrow 0}\frac{P_tf(x)-f(x)}{t}, \quad x\in\real^d, f\in\mathscr{D}(\ccA),
\end{align*} 
where
\begin{align*}
\mathscr{D}(\ccA)&:=\left\{f\in \cC_\infty(\real^d): \lim_{t\downarrow 0}\frac{P_tf(x)-f(x)}{t} \text{ exists in } \|\cdot\|_\infty \right \},
\end{align*}
is called the \textbf{Feller generator} of \((X_t)_{t\geq0}\).  Clearly, $\mathscr{D}(\ccA) \subset \domain$ and \(\ccA= \cA|_{\mathscr{D}(\ccA)}\) since uniform convergence implies pointwise convergence. \\
A Feller process is called \textbf{rich} if \(\test\subset\mathscr{D}(\ccA)\). In this case,  by the Courrège-von Waldenfels theorem \cite[Thm. 2.21]{bottcher}, the operator  \((\ccA|_{\test},\test)\) is a Lévy-type operator. In some sources, e.g. \cite{sandric}, rich Feller processes are therefore also called \textbf{Lévy-type processes}. However, since other authors define Lévy-type processes via its probabilistic symbol, cf. \cite{bottcher,schnurr2}, yielding a larger class of processes, or as a solution to martingale problems of Lévy-type operators, cf. \cite{wang, wang2}, we avoid this terminology. \\

Let \(\eta\) be a (positive) measure on \(\real^d\). Then $\eta$ is called \textbf{invariant} for the Markov process \((X_t)_{t\geq0}\) in \(\real^d\) if and only if
\begin{align}\label{eq_definv}
\int_{\real^d}p_t(x,B)\eta(\diff x) = \eta(B) \quad \text{for all }t\geq0, B\in\cB(\real^d).
\end{align}
If \((X_t)_{t\geq0}\) is a Feller process with Feller generator \((\ccA,\mathscr{D}(\ccA))\) and \(\eta\) is finite, it is further well-known, cf. [Liggett, Thm 3.37],  that \eqref{eq_definv} is equivalent to
\begin{align}\label{eq_equiv_inv}
\int_{\real^d} \ccA f(x)\eta(\diff x)=0\quad \text{for all }f\in D,
\end{align}
where \(D\) is a core of \((\ccA,\mathscr{D}(\ccA))\). Together with Theorem \ref{thm_distEQ} this implies the following corollary.

\begin{corollary}\label{cor_richFellerprocess}
	Let \((X_t)_{t\geq0}\) be a rich Feller process such that its generator \((\ccA,\mathscr{D}(\ccA))\) is a Lévy-type operator in $\real^d$ with characteristic triplet \((a,b,\Pi)\) and convenient decomposition \(\pi=(\nu,\mu,\rho)\) of the Lévy kernel \(\Pi\). Let \(\eta\in\cM(\cA,\pi)\) be finite. 
	If a core of \((\ccA,\mathscr{D}(\ccA))\) is contained in \(\test\), and \(\eta\) solves \eqref{eq_distributional}, then \(\eta\) is invariant for \((X_t)_{t\geq0}\).
	Conversely, if $\eta$ is invariant for \((X_t)_{t\geq0}\), then it solves \eqref{eq_distributional}.
\end{corollary}

With the next corollary we demonstrate that stronger assumptions on the characteristic triplet yield a larger space in which a core is to be found. Indeed, simply observe that each step in the proof of Theorem \ref{thm_distEQ} equally works if \(f\in\cC_c^2(\real^d)\) instead of \(f\in\test\), if we assume that all occurring distributions are of order 2.

\begin{corollary} \label{cor_necsufFeller}
	Let \((X_t)_{t\geq0}\) be a Feller process whose Feller generator \((\ccA,\mathscr{D}(\ccA))\) is a Lévy-type operator in \(\real^d\) with characteristic triplet \((a,b,\Pi)\) and convenient decomposition \(\pi=(\nu,\mu,\rho)\) of $\Pi$. Let \(\eta\) be a measure on \(\real^d\) such that \(\nabla \cdot ((a+a_\pi)\eta),~{\frac12 \nabla \cdot b\nabla\eta},\) \(U^\pi(\eta), {\nabla\cdot V^\pi(\eta)}\in\cD'^2\).
	If a core of \((\ccA,\mathscr{D}(\ccA))\) is contained in \(\cC^2_c(\real^d)\) and \(\eta\) solves \eqref{eq_distributional}, 
	then it is invariant for \((X_t)_{t\geq0}\).
\end{corollary}

Since we aim to use a relation as in \eqref{eq_equiv_inv} also in the context of processes that do not necessarily have the Feller properties, we state the following lemma.

\begin{lemma}\label{lem_invlaw}
Let \((X_t)_{t\geq0}\) be a Markov process in \(\real^d\) with pointwise generator \((\cA,\domain)\), and let \(\eta\) be an invariant measure for \((X_t)_{t\geq0}\). For all functions \(f\in\domain\) for which there exists \(g\in L^1(\eta)\) and \(T>0\) such that \(\left|\frac{P_tf(x)-f(x)}{t}\right| \leq g(x)\) for all \(x\in\real^d\) and \(t\leq T\), it holds
\begin{align}
\int_{\real^d} \cA f(x)\eta(\diff x)=0.
\end{align}
\end{lemma}
\begin{proof}
First of all note that \eqref{eq_definv} is equivalent to \(\int_{\real^d} P_tf(x)\eta(\diff x) = \int_{\real^d} f(x)\eta(\diff x)\) for all \(f\in\cC_\infty(\real^d)\). Thus 
\begin{align*}
0=\lim_{ t\downarrow 0} \frac1t\int_{\real^d}  (P_tf(x)-f(x))\eta(\diff x)=\int_{\real^d}  \lim_{t\downarrow 0}\frac{P_tf(x)-f(x)}{t}\eta(\diff x)=\int_{\real^d} \cA f(x)\eta(\diff x) 
\end{align*}
for all \(f\) as in the statement of the lemma, since the assumption on \(f\) ensures the applicability of Lebesgue's theorem in the second step.
\end{proof}

Combining Theorem \ref{thm_distEQ}, the definition of invariant measures, and Lemma \ref{lem_invlaw} yields the following corollary. 

\begin{corollary}\label{cor_necsuf}
Let \((X_t)_{t\geq0}\) be a Markov process whose pointwise generator \((\cA,\domain)\) is a Lévy-type operator with characteristic triplet \((a,b,\Pi)\) and convenient decomposition \(\pi=(\nu,\mu,\rho)\) of the Lévy kernel \(\Pi\). Assume \(\test\in\domain\) and let \(\eta\in\cM(\cA,\pi)\). If \(\eta\) is invariant for \((X_t)_{t\geq0}\), and all \(f\in\test\) fulfill the additional assumption of Lemma \ref{lem_invlaw}, then \(\eta\) solves \eqref{eq_distributional}.
\end{corollary}

\section{Lévy-driven SDEs}\label{sec_sde}

Markov processes associated to Lévy-type operators as discussed in the previous section often appear as solutions of Lévy-driven SDEs; cf. \cite{albeverio2,behme,bottcher,kuhn}. In this section we therefore consider SDEs of the form
\begin{align}\label{eq_sde}
\diff X_t = \phi(X_{t-})  \diff L_t
\end{align}
where \(\phi: \real^d \to \real^{d\times n}\) is a function and \((L_t)_{t\geq0}\) is a Lévy process in \(\real^n\) with characteristic triplet \((\gamma,\Sigma,\Upsilon)\). Note that solutions of \eqref{eq_sde}, if existent, are in general not Feller processes. In fact, a solution to \eqref{eq_sde} may even fail to be a Markov process, cf. \cite[Ex. 3.10 and subsequent Remark]{cherny1858singular}.  \\ 
In Section \ref{secSDEsublinear} we discuss a quite general setting in which solutions of \eqref{eq_sde} are Feller processes and Corollary \ref{cor_necsufFeller} is applicable. Thereafter we present an explicit example in Section \ref{sec_superlinear} where the solution of \eqref{eq_sde} is not necessarily a Feller process but Corollary \ref{cor_necsuf} is applicable, and we can still derive an equation for the invariant measure. 
  
\subsection{Lipschitz continuous coefficients with sublinear growth}\label{secSDEsublinear}

In order to apply our results we first need to ensure that \eqref{eq_sde} has a (unique) Markov solution \(\X\) with pointwise generator \((\cA,\domain)\) such that \(\test\subset\domain\). To this end, we assume that
\begin{enumerate}
\item[(\textbf{a})] \(\phi\) is Lipschitz continuous,
\item[(\textbf{b})] there exists \(C>0\) such that \(|\phi(x)|\leq C(1+|x|)\) for all \(x\in\real^d\), and
\item[(\textbf{c})] \(\Upsilon(\{y\in\real^n; \phi(x)y\in B(-x,r)\}) \xrightarrow[] {|x|\to\infty} 0\) for all \(r>0\).
\end{enumerate}
These assumptions supply \(\X\) with additional nice properties as stated in the following theorem which is a summary of \cite[Thms. 2.46 \& 2.47]{schnurr2009symbol} and \cite[Thm. 1.1]{kuhn}.
\begin{theorem}\label{thm_sol}
If (\textbf{a}) - (\textbf{c}) hold, then \eqref{eq_sde} possesses a unique strong solution \(\X\). Moreover, \(\X\) is a Feller process with Feller generator \((\ccA,\mathscr{D}(\ccA))\) with $\cC^2_c(\real^d) \subset\mathscr{D}(\ccA)$, and such that
\begin{align}
\ccA f(x) &= \left(\phi(x) \gamma - \frac12 \nabla \cdot(\phi(x)\Sigma \phi(x)^T) \right) \cdot \nabla f(x)  + \frac12 \nabla\cdot(\phi(x)\Sigma \phi(x)^T) \nabla f(x) \nonumber \\
&\quad+ \int_{\real^d}(f(x+y)-f(x) - \nabla f(x) \cdot y\bone_{\{|y|<1\}}) \Upsilon_\phi(x,\diff y),\label{eq_constrainedgenerator}
\end{align}
for all \(f\in\test\subset\mathscr{D}(\ccA)\), where \(\Upsilon_\phi\) is a Lévy kernel in \(\real^d\) defined by
\begin{align*}
\Upsilon_\phi(x,B):=\begin{cases}\Upsilon(\{y\in\real^n; \phi(x)y \in B\}), &\phi(x)\neq0,\\ 0, &\phi(x)=0,\end{cases} \quad x\in\real^d, B\in\cB(\real^d).
\end{align*}
\end{theorem}

In view of the discussion in Section \ref{sec_cond}, the main obstacle in the application of Theorem \ref{thm_distEQ} to the solution of the SDE \eqref{eq_sde} is to check, whether a candidate measure \(\eta\) is in \(\cM(\cA,\pi)\). Note that if \(\Upsilon\) is finite and $\phi\in \cC^1$, it follows from Corollaries \ref{cor_etainM} and \ref{cor_sdefinite} that all finite measures \(\eta\), and in particular all probability measures, are contained in \(\cM(\ccA,\pi)\).

In general, a combination of Corollary \ref{cor_necsufFeller} and Theorem \ref{thm_sol} immediately yields the following main result of this section. 

\begin{theorem}\label{cor_sde_sol} 
In the setting of Theorem \ref{thm_sol}, let \(\pi=(\nu,\mu,\rho)\) be a convenient decomposition of \(\Upsilon_\phi\). Let $\eta \in\cM(\ccA,\pi)$ be finite. If a core of \((\ccA,\mathscr{D}(\ccA))\) is contained in \(\test\),  then \(\eta\) is invariant for \(\X\) if and only if 
\begin{align}
0=&- \nabla \cdot \left(\left(\phi(x) \gamma - \frac12 \nabla\left(\phi(x)\Sigma \phi(x)^T\right) +a_\pi(x)\right)\eta\right)+ \frac12 \nabla \cdot (\phi(x)\Sigma \phi(x)^T)\nabla\eta \nonumber \\
&+ U^\pi(\eta) +\nabla \cdot V^\pi(\eta) + \hess \cdot W^\pi(\eta) \label{eq_sde_disteq} 
\end{align}
on \(\real^d\), where \(a_\pi(x)= -\int_{|y|<1}y\nu(x,\diff y)- \int_{|y|<1}y\mu(x,\diff y)+\int_{|y|\geq1}y\rho(x,\diff y)\), and $U^\pi(\eta)$, $V^\pi(\eta)$, and \(W^\pi(\eta)\) denote the jump functionals of \(\Upsilon_\phi\) w.r.t.   \(\pi\) and    \(\eta\).
\end{theorem}

If \(\X\) is a one-dimensional process the distributional equation  \eqref{eq_sde_disteq} can be transformed into a Volterra-Fredholm integral equation as we have seen in Section \ref{sec_onedim}. In the following we present this in a special case for which the respective integral equation turns out to be particularly nice.\\

 Consider the SDE
\begin{align}\label{eq_simplesde}
\diff X_t = \varphi(X_{t-}) \diff B_t + \diff S_t,
\end{align}
where \(\varphi: \real\to \real^{1\times n}\) is a function fulfilling (\textbf{a}) and (\textbf{b}), \((B_t)_{t\geq0}\) is an \(n\)-dimensional Brownian motion with drift having the characteristic triplet \((\gamma,\Sigma,0)\), and \((S_t)_{t\geq0}\) is a one-dimensional pure-jump Lévy process with characteristic exponent
\begin{align*}
\psi(\xi)= \int_{\real^d}(1-\ee^{iy\xi})\mu(\diff y) + \int_{\real^d}(1-\ee^{iy\xi} + i y\xi)\rho(\diff y),
\end{align*}
where \(\mu\) and \(\rho\) are Lévy measures such that \(\int_{y\neq0}(1\wedge|y|)\mu(\diff y)+\int_{y\neq0}(|y|\wedge|y|^2)\rho(\diff y)<\infty\). In particular, this implies that we can set $a_\pi(x)\equiv 0$. \\Note that (\textbf{c}) is automatically fulfilled for \eqref{eq_simplesde} as the jumps of \(\X\) are space-homogeneous.\\
In this case, the Volterra-Fredholm integral equation \eqref{eq_vfie_template} for an infinitesimally invariant measure of the associated Lévy-type operator $(\ccA,\mathscr{D}(\ccA))$ takes the form
\begin{align}
\frac{1}{2}\varphi(x) \Sigma \varphi(x)^T  \eta(\diff x) &= \Big( c_1 x + c_2 - \int_\real \int_{(0,x]} \sgn (y-z) \tilde{\mu}(z,y-z) \diff y \, \eta(\diff z) \nonumber \\
	&\quad  - \int_\real \tilde{\bar{\rho}} (z,x-z) \eta(\diff z) + \int_{(0,x]} \varphi(z) \cdot \gamma \, \eta(\diff z)  \Big) \diff x. \label{eq_VFSDE}
\end{align}
Specifying some of the constants then yields rather simple equations as illustrated in the following corollary.

\begin{corollary}\label{cor_2convo} 
Let \(\X\) be the unique solution of \eqref{eq_simplesde} and let \(\eta\) be a finite measure.
\begin{enumerate}
\item Assume \(\gamma=0\) and \(\mu=0\). If \(\eta\) is invariant for \(\X\), then there exist constants \(c_1,c_2\in\real\) such that
\begin{align*}
\frac{1}{2} \varphi(x) \Sigma \varphi(x)^T \eta(\diff x)= \big(c_1x+c_2 -(\tilde{\overline\rho}*\eta)(x)\big)\diff x
\end{align*}
on \(\real\). Additionally, if $\varphi(x) \Sigma \varphi(x)^T \neq 0$ for all $x\in\real$, then $\eta$ is absolutely continuous.
\item Assume \(\sigma^2=0\) and \(\rho=0\).  If \(\eta\) is invariant for \(\X\), then there exists a constant \(c_3\in\real\) such that
\begin{align*}
\varphi(x) \cdot \gamma \,\eta(\diff x)= \big(c_3+((\sgn(\cdot)\tilde\mu)*\eta)(x)\big)\diff x
\end{align*}
on \(\real\). Additionally, if $\varphi(x) \cdot \gamma\neq 0$ for all $x\in\real$, then $\eta$ is absolutely continuous.
\end{enumerate}
\end{corollary}
\begin{proof}
Note that in both cases \(\eta\in\cM(\ccA,\pi)\) can be shown with Corollaries \ref{cor_etainM} and \ref{cor_constantPi}, while the given equations are special cases of \eqref{eq_VFSDE} except that for (ii) \eqref{eq_sde_disteq} needs to be integrated only once to obtain the desired result. That an invariant measure solves the respective equation is a consequence of Theorem \ref{thm_vfie}, Corollary \ref{cor_richFellerprocess}, and Theorem \ref{thm_sol}. The absolute continuity in both cases follows by Corollary \ref{cor_abscont}.
\end{proof}
 
\begin{example}[Stable process + space dependent drift]
In \eqref{eq_simplesde}, let \(\Sigma=0\) and let \((S_t)_{t\geq0}\) be a spectrally positive stable process with parameter \(\alpha\in(0,1)\), i.e. \(\rho=0\) and \(\mu(\diff x)=\bone_{\{x>0\}}x^{-(1+\alpha)}\diff x\), cf. \cite[Thm. 14.3]{sato2nd}. Hence, we are in case (ii) of Corollary \ref{cor_2convo} with \(\tilde\mu(x)=\bone_{\{x>0\}}\alpha^{-1} x^{-\alpha}\). Suppose that \(\varphi(0)\cdot \gamma =0\), and \(\varphi(x) \cdot \gamma <0\) for all \(x>0\) and let \(X_0>0\). Then \(X_t>0\) for all \(t\geq0\) and by  Corollary \ref{cor_abscont} any invariant probability measure $\eta$ of $(X_t)_{t\geq 0}$ is absolutely continuous with density  \(\eta(x)\) on \((0,\infty)\).\\
By Corollary \ref{cor_2convo}(ii) this density solves for some constant \(c\in\real\)
\begin{align}\label{eq_frac}
c+ (\varphi(x)\cdot \gamma) \eta(x) = \tilde\mu*\eta(x) = \frac{1}{\alpha} \int_{(0,x)} \frac{\eta(y)}{(x-y)^\alpha} \diff y= \frac{\Gamma(1-\alpha)}{\alpha} (I^{1-\alpha}_{0+} \eta)(x), \quad x>0,
\end{align}
with the  Riemann-Liouville fractional integral
\begin{align*}
(I_{0+}^{\alpha}\eta)(x):=\frac1{\Gamma(\alpha)} \int_{(0,x)} \frac{\eta(y)}{(x-y)^{1-\alpha}}\diff y.
\end{align*}
Thus, solving \eqref{eq_frac} amounts to solving a fractional integral equation.
\end{example}

\subsection{Superlinear growth}\label{sec_superlinear}

In this section we consider an SDE that does not fulfill assumption (\textbf{b}) of the previous subsection and hence its solution is not necessarily a Feller process. Nonetheless, we are still able to derive a necessary criterion for an invariant probability measure of a solution of the SDE below. Thus the growth condition (\textbf{b}) is, in general, not necessary to obtain an equation for the invariant measure of a Lévy-driven SDE.\\

Consider the SDE
\begin{align}\label{sde_fastgrowth}
\diff X_t^x = \varphi_1(X^x_{t-}) \diff B_t  + \varphi_2(X^x_{t-})\diff N_t, \quad X_0^x=x\in\real,
\end{align}
with  \(0\leq\varphi_2\in\testo\), \(\varphi_1\in\cC^\infty(\real)\) such that \(\varphi_1(x)>C(1+|x|^{1+\alpha})\) for some \(C,\alpha>0\), \((B_t)_{t\geq0}\)  is a standard Brownian motion in \(\real\), and \((N_t)_{t\geq0}\) a Poisson process with intensity \(\lambda=1\), independent of \((B_t)_{t\geq0}\).

\begin{theorem}\label{thm_super}
Equation \eqref{sde_fastgrowth} has a unique strong solution \(\X\), and, if a probability measure \(\eta\) is invariant for \(\X\), it is absolutely continuous, i.e. \(\eta(\diff x)=\eta(x)\diff x\), where the density $\eta(\cdot)$ solves the equation
\begin{align}\label{eq_sol_of_suplinear}
\frac12 \left(\varphi_1^2(z)\eta(z)\right)' = -\int_\real \bone_{\{x<z<x+\varphi_2(x)\}}\eta(x)\diff x.
\end{align}
Moreover, there exist \(M,C_1,C_2 > 0\)  such that
\begin{align*}
\eta(x)=\begin{cases}
\frac{C_1}{\varphi_1^2(x)}, & x>M,\\
\frac{C_2}{\varphi_1^2(x)}, & x<-M.
\end{cases}
\end{align*} 
\end{theorem}

\begin{remark}
The behaviour of \(\eta\) on \([-M,M]^c\) has a particular consequence concerning a closely related necessary criterion for invariance: Denote by \(q: \real\times\real\to\complex\) the \textbf{symbol} of \(\cA\), i.e.
\begin{align*}
\cA f(x)= -\int_\real q(x,\xi) \hat f(\xi) \ee^{ix\xi}\diff \xi, \quad f\in\testo,
\end{align*}
where \(\hat f\) denotes the Fourier transform of \(f\). In \cite[Thm. 3.1]{behme} it is shown that, for a  probability measure \(\eta\)  to be invariant for \(\X\), it is necessary that
\begin{align*}
\int_\real \ee^{ix\xi}p(x,\xi)\eta(\diff x) = 0 \quad \text{for all } \xi\in\real.
\end{align*}
 However, a condition for this to be true is that \(\int_\real |q(x,\xi)|\eta(\diff x)<\infty\) for all \(\xi\in\real\). From \eqref{eq_g} it is quickly derived that, in our case, \(q(x,\xi)= \frac12\varphi_1(x)^2\xi^2 + (1-\ee^{i\varphi_2(x)\xi})\), and thus, \(\int_\real |q(x,\xi)|\eta(\diff x)\) does not converge. Hence, the results from \cite{behme} are not applicable here.
\end{remark}

\begin{proof}[Proof of Theorem \ref{thm_super}] 
We divide the proof into several steps.
\begin{itemize}
\item[(1)] \textit{Equation \eqref{sde_fastgrowth} has a unique strong solution \(\X\) for all \(t\geq0\).}
\end{itemize}
Let $x\in \real$ be fixed and let \((X_t^x)_{t\geq0}\) denote a solution of the SDE \eqref{sde_fastgrowth} with $X_0=x$ as long as it exists. Indeed, by \cite[Thm. V.38]{protter} there exists a unique (strong) solution \((X_t^x)_{t\geq0}\) of \eqref{sde_fastgrowth} up to an \(x\)-dependent stopping time \(\tau_x\), called the \textbf{explosion time} of \((X_t^x)_{t\geq0}\), for which \(\limsup_{t\to\tau_x} |X_t^x| = \infty\) on \(\{\tau_x<\infty\}\). To prove the claim we thus intend to show \(\tau_x=\infty\) \(\PP_x\)-a.s. \\
Since \(\varphi_2\in\testo\) there exists \(M'>0\) such that \(\supp(\varphi_2)\subset [-M',M']\). Let \(M:=1+M'+~\|\varphi_2\|_\infty\) and consider a function \(h\in\testo\) with \(h(x)=1\) on \([-M,M]\). Let the auxiliary process \((Y_t^x)_{t\geq0}\) solve
\begin{align*}
	\diff Y^x_t =  h(Y^x_{t-}) \varphi_1(Y^x_{t-}) \diff B_t +  h(Y^x_{t-})\varphi_2(Y^x_{t-})\diff N_t,  \quad  Y_0^x=x  \in[-M,M].
\end{align*}
By \cite[Thm. V.32]{protter}, \((Y_t^x)_{t\geq0}\) exists, and is a (pathwise) unique and non-explosive strong Markov process, because the coefficients $h\varphi_1$ and $h\varphi_2$ are globally Lipschitz continuous. \\
Let \(x\in[-M,M]\). Then \(X^x_t=Y^x_t\) a.s. for all \(t\in[0,\vartheta)\) where \(\vartheta:=\inf\{s\geq 0: |X^x_s|\geq M\}\) by the uniqueness of solutions. Moreover \(\vartheta<\tau_x\) since otherwise \(\limsup_{t\to\tau_x} |X_t^x|\leq M\), and in particular if \(\vartheta=\infty\) then \(\tau_x=\infty\). Further, as \(Y^x_t\) and \(X^x_t\) cannot enter \([-M,M]^c\) via jumps, they creep across the barrier $\{-M,M\}$ and we have \(X^x_{\vartheta}=Y^x_{\vartheta}\in\{-M,M\}\) a.s. on \(\{\vartheta<\infty\}\).\\
Conditioning on \(\{\vartheta<\infty\}\) we obtain
\begin{align*}
X^x_{\vartheta + t} &= x + \int_0^{\vartheta+t} \varphi_1(X^x_{s-}) \diff B_s + \int_0^{\vartheta+t} \varphi_2(X^x_{s-})\diff N_t\\
&= X^x_{\vartheta} + \int_{\vartheta}^{\vartheta+t}\varphi_1(X^x_{s-}) \diff B_s + \int_{\vartheta}^{\vartheta+t} \varphi_2(X^x_{s-})\diff N_t \\
&= X^x_{\vartheta} + \int_0^t\varphi_1(X^x_{(\vartheta+s)-}) \diff (B_{\vartheta+s}-B_{\vartheta})
\end{align*}
for all \(t+\vartheta <\theta:=\inf \{s\geq \vartheta: |X^x_{s}|\leq M' \}\). Moreover, conditioned on \(\{\vartheta<\infty\}\), we have that \((B_{\vartheta+s}-B_{\vartheta})_{s\geq0}\) is a Brownian motion independent of \(\sigma(X_s; s\leq \vartheta)\). We thus consider the equation
\begin{align}\label{sde_z}
Z^x_t =  X^x_{\vartheta} + \int_{(0,t]} \varphi_1(Z^x_{s-}) \diff B_s.
\end{align}
By the Engelbert-Schmidt theorem, see \cite[Thms. 5.5.4 and 5.5.7]{karatzas}, there exists a unique weak solution \((Z^x_t)_{t\geq0}\) of \eqref{sde_z} which is non-explosive. Since $X_t^x$ is a unique strong solution up to the explosion time, there exists an equivalent version of \((Z^x_t)_{t\geq0}\) such that \(X^x_{\vartheta+t}= Z_t^x\) a.s. for all \(t\in[0,\theta-\vartheta)\). Again by continuity we obtain \(X^x_{\vartheta+\theta} = Z^x_{\theta}\in\{-M',M'\}\) a.s. on \(\{\theta<\infty\}\). \\
Now, if \(\theta=\infty\) we have \(\tau_x=\infty\) due to \((Z_t)_{t\geq0}\) being non-explosive. Since \(X^x_{\vartheta+\theta}\in[-M,M]\) on \(\{\theta<\infty\}\) we can reiterate this argument to obtain a sequence \((\vartheta_n,\theta_n)_{n\in\nat_0}\) with \(\vartheta_n\leq \theta_n\leq\vartheta_{n+1}\) for all \(n\in\nat_0\), defined by \(\vartheta_0=\vartheta\), $\theta_0=\theta$, and
\begin{align*}
\theta_n&:=\inf\{t>\sigma_{n}: |X^x_t|\leq M'\},\\
\vartheta_{n+1}&:=\inf\{t>\theta_{n}: |X^x_t|\geq M\},
\end{align*}
for all \(n\geq1\).\\
Due to the fact that \((B_t)_{t\geq0}\) and \((N_t)_{t\geq0}\) have stationary and independent increments, and as $X_t^x$ always creeps into $[-M',M']$ or out of $[M,M]$, it follows that the durations of the excursions \((\vartheta_n-\theta_{n-1})_{n\in \nat}\) form a sequence of independent identically distributed random variables. Further, there exist \(\epsilon,\delta>0\) such that \(\PP(\vartheta_n-\theta_{n-1}>\epsilon)>\delta\) for all \(n\in\nat\) as otherwise \(\vartheta_n=\theta_n\) a.s. for all \(n\in\nat\).
By the same arguments as before it holds \(\vartheta_n,\theta_n<\tau_x\) for all \(n\in\nat_0\) if \(\tau_x<\infty\).  Fix \(N\in\nat\), then
\begin{align*}
\PP(\tau_x\leq N\epsilon)\leq \PP\left(\sum_{n=1}^m (\vartheta_n-\theta_{n-1})\leq N\epsilon\right) \leq \prod_{n=1}^{(m-N) \vee 0} \PP(\vartheta_n-\theta_{n-1}\leq\epsilon) <  (1-\delta)^{(m-N) \vee 0}
\end{align*}
for all \(m\in\nat\), since at most \(N\) out of the variables \((\vartheta_n-\theta_{n-1})_{n=1,\ldots,n}\) are allowed to be larger than \(\epsilon\). But letting \(m\to\infty\), the right-hand side of this inequality tends to \(0\). Since \(N\in\nat\) was arbitrary this implies that \(\{\tau_x<\infty\}\) is a \(\PP^x\)-null set for \(x\in[-M,M]\).\\ Clearly, the same arguments work for \(x\in[-M,M]^c\) as well by simply omitting the first stopping time \(\vartheta\).
\begin{itemize}
\item[(2)] \textit{The domain of the pointwise generator of \(\X\) contains \(\testo\).}
\end{itemize}
Let \(f\in\testo\) and fix \(x\in\real\) and \(t\in\real_+\). By Itô's formula and \eqref{sde_fastgrowth} we obtain
\begin{align*}
\lefteqn{\EE^x[f(X_t)-f(x)] }\\ 
&= \EE^x \Big[\int_{ (0,t]} f'(X_{s-})\varphi_1(X_{s-})\diff B_s\Big] + \EE^x\Big[ \int_{(0,t]} f'(X_{s-})\varphi_2(X_{s-})\diff N_s\Big]\\
& \quad + \frac12 \EE^x\Big[ \int_{(0,t]} f''(X_{s-})\varphi_1(X_{s-})^2\diff [B,B]^c_s \Big]\\
& \quad + \EE^x\Big[ \int_{\real}\int_{(0,t]} \left( f(X_{s-}+\varphi_2(X_{s-})y) -f (X_{s-}) - f'(X_{s-})\varphi_2(X_{s-})y \right) \mu^N(\cdot,\diff s,\diff y)\Big], 
\end{align*}
where \(\mu^N(\omega;A,B)\), $A\in\cB(\real_+), B\in\cB(\real)$ denotes the jump measure of \((N_t)_{t\geq0}\). 
As $f'\varphi_1$ is bounded, the first summand is the expected value of a martingale and vanishes. Moreover, we note that  \([B,B]_s^c=[B,B]_s=s\). Thus, merging the integral term w.r.t. \(N_s\) with the last term we obtain
\begin{align*}
\EE^x[f(X_t)-f(x)]&= \frac12 \EE^x \Big[\int_{(0,t]} f''(X_{s-})\varphi_1(X_{s-})^2\diff s \Big] \\
& \quad + \EE^x \Big[ \int_{\real}\int_{(0,t]} \left( f(X_{s-}+\varphi_2(X_{s-})y) -f (X_{s-}) \right) \mu^N(\cdot,\diff s,\diff y)\Big].
\end{align*}
Since \(f,\varphi_2\in\cC_c^\infty(\real)\) the integrand of the second integral is bounded. We can therefore write the integral under the expectation in terms of the compensator of \(\mu^N\) and swap the order of integration which yields
\begin{align*}
\lefteqn{\EE^x \Big[ \int_{\real}\int_{(0,t]} \left( f(X_{s-}+\varphi_2(X_{s-})y) -f (X_{s-}) \right) \mu^N(\cdot,\diff s,\diff y)\Big]} \\
&= 	\EE^x \Big[  \int_{(0,t]} \int_{\real}\left( f(X_{s-}+\varphi_2(X_{s-})y) -f (X_{s-}) \right) \delta_1(\diff y)\,\diff s \Big]\\
&= \EE^x \Big[  \int_{(0,t]} \left( f(X_{s-}+\varphi_2(X_{s-})) -f (X_{s-}) \right) \diff s \Big].
\end{align*}
Hence, setting 
\begin{align}\label{eq_g}
	g(x):=\frac12 f''(x)\varphi_1(x)^2+f(x+\varphi_2(x))-f(x), \quad x\in\real,
\end{align}
we observe that 
\begin{align*}
\EE^x[f(X_t)-f(x)]=\EE^x \Big[ \int_{(0,t]} g(X_{s-})\diff s \Big]=\EE^x\Big[ \int_{(0,t]} g(X_{s})\diff s \Big]
\end{align*}
Clearly, \(g\in\testo\). Since the process \(g(X_s)\) is bounded and right continuous at zero we obtain
\begin{align}\label{eq_lim}
 \lim_{t\downarrow 0}\frac1t\EE^x\left[\int_{(0,t]} g(X_s)\diff s\right] = g(x),\quad  x\in\real,
\end{align}
by \cite[Lem. 2.51]{schnurr2009symbol}. Therefore \(f\in\domain\) and \(\cA f= g\).
\begin{itemize}
\item[(3)]\textit{For any $f\in \cC_c^\infty(\real)$, the function \((t,x)\mapsto  \frac1t (P_t f(x) - f(x))\) is bounded.}
\end{itemize}
We know that \(g\) as defined in \eqref{eq_g} is a test function and therefore bounded. Thus 
\begin{align*}
\frac1t |P_t f(x) - f(x)|\overset{(2)}{=}\frac1t\left|\EE^x\left[\int_0^t g(X_s)\diff s\right]\right|\leq\frac1t\int_0^t \EE^x\left[\left|g(X_s)\right|\right] \diff s \leq \|g\|_\infty <\infty.
\end{align*}
\begin{itemize}
\item[(4)]\textit{Derive a necessary criterion for \(\eta\).}
\end{itemize}
As $\eta$ is assumed to be a  probability measure, all bounded functions are contained in \(L^1(\eta)\). Thus the condition in Lemma \ref{lem_invlaw} is fulfilled for all \(f\in\testo\), and if $\eta$ is invariant, necessarily for all \(f\in\testo\)
\begin{align}\label{eq_invspec}
\int_\real \cA f(x)\eta(\diff x)=0.
\end{align}
Hereby, for all $f\in\testo$, the operator \((\cA,\domain)\) can be rewritten in the form of a Lévy-type operator as in Lemma \ref{lem_genrep}, namely 
\begin{align*}
\cA f(x)= g(x) 
&= - \varphi_1(x) \varphi_1'(x) f'(x) + \frac12 \left(\varphi_1^2(x) f'(x) \right)' + \int_{\real} \sgn(y)f'(x+y)\overline \mu(x,\diff y),
\end{align*}
with  $$\overline\mu(x,\diff y):= \bone_{(0,\varphi_2(x))}(y)\diff y.$$
The corresponding jump functional is therefore given by 
\begin{align*}
\langle V(\eta), f \rangle &= \int_\real \int_{(0,\varphi_2(x))} f(x+y)\diff y \,\eta(\diff x) \\
&= \int_\real f(z) \int_\real \bone_{\{x<z< x+\varphi_2(x)\}} \eta(\diff x)\,\diff z =: \int_\real f(z) V(\eta,z)\diff z,
\end{align*}
which, in this case, is a regular distribution, associated with the locally integrable function \(V(\eta,\cdot)\). By Theorem \ref{thm_distEQ} \(\eta\) therefore solves
\begin{align}\label{eq_super_eta}
\frac 12\left({\varphi_1^2} \eta\right)'' = - V(\eta,\cdot)'
\end{align}
in the distributional sense. 
\begin{itemize}
\item[(5)]\textit{Any invariant probability measure \(\eta\) is absolutely continuous, and for \(|x|>M\) its density is inverse proportional to \(\varphi_1^2\).}
\end{itemize}
The absolute continuity of \(\eta\) follows from Corollary \ref{cor_abscont} since $\varphi_1(x)\neq 0$ on $\real$. \\ 
Setting $\eta(\diff x) = \eta(x) \diff x$ and integrating  \eqref{eq_super_eta} twice yields
\begin{align*}
\frac12{\varphi_1^2(x)} \eta(x) = - \int_{(0,x]} V(\eta,y)\diff y + c_1x+c_2.
\end{align*}
As $\supp(\varphi_2)\subset [-M',M']$ we observe that for \(y\notin[-M,M]\) it holds \(V(\eta,y)=0\) and in particular \(\int_{(0,x]} V(\eta,y)\diff y\) is constant outside of \([-M,M]\). From this we conclude that, if \(c_1\neq0\), the right-hand side would turn be negative for either \(x\ll -M\) or \(x\gg M\). As the left-hand side is non-negative for all \(x\in\real\) this implies \(c_1=0\). Differentiating again on both sides yields \eqref{eq_sol_of_suplinear}. \\ 
Moreover, it follows that there exist constants \(C_{1,2} \geq 0\) such that
\begin{align*}
\eta(x) = \begin{cases}
\frac{C_1}{\varphi_1^2(x)}, & x>M,\\
\frac{C_2}{\varphi_1^2(x)}, & x<-M.
\end{cases}
\end{align*} 
Lastly, for any compact interval \([a,b]\subset\real\) the process \((X^x_t)_{t\geq0}\) with \(x\in(a,b)\), conditioned on the event that no jump occurs until the exit time \(\tau_{a,b}:=\inf\{t>0 : X^x_t\notin[a,b]\}\), behaves like a regular diffusion on natural scale for \(t\in[0,\tau_{a,b}]\), see \cite[Chap. 33]{kallenberg} for details. As this implies that any open interval can be reached in finite time with positive probability, it follows that \(C_{1,2}>0\).  \\

This finishes the proof.
\end{proof}

\section{The adjoint of a Lévy-type operator}\label{sec_Adj}
\setcounter{equation}{0}

In this final section we discuss the implications of our results on the adjoint operator of a Lévy-type operator.

Let \(\cE_1,\cE_2\) be two Banach spaces over \(\real\) and denote their dual spaces by \(\cE_1'\) and \(\cE_2'\), respectively. Given a linear operator \(\cA:\domain\to\cE_2\) with \(\domain\subset\cE_1\) dense, the \textbf{adjoint operator} \((\cA',\mathscr{D}(\cA'))\) of \((\cA,\domain)\) is defined as the linear operator \(\cA': \mathscr{D}(\cA') \to \cE_1'\)  with domain
\begin{align}\label{def_dom_adj}
 \mathscr{D}(\cA')&:= \{\phi\in\cE_2': \exists c\geq 0 :  |\phi (\cA f)| \leq c\|f\|_{\cE_1}  ~\forall f\in\domain\} \subset \cE_2'
\end{align}
and such that 
\begin{align}\label{eq_adj_ident}
  (\cA'\phi)(f) = \phi( \cA f)
\end{align}
holds for all \(f\in\domain\) and \(\phi\in\mathscr{D}(\cA')\). 

The relation \eqref{eq_adj_ident} is the reason why sometimes adjoint operators are used to characterize infinitesimally invariant measures of Lévy-type operators; cf \cite{albeverio2,albeverio} for Lévy-type operators acting on \(L^2(\real^d)\), or \cite{sato} for the generators of OUT processes. Indeed, let \((\cA,\domain)\) be a Lévy-type operator with adjoint \((\cA',\mathscr{D}(\cA'))\), assume \(\test\subset\domain\) and let \(\eta\in\mathscr{D}(\cA')\) be an infinitesimally invariant measure of \((\cA,\domain)\). Then
\begin{align*}
	 (\cA'\eta)(f) = 0  \quad\text{for all }f\in\test,
\end{align*} 
i.e. \(\eta\) is a weak solution of
\begin{align}\label{def_adj_inv}
\cA'\eta = 0.
\end{align}

Our main result, Theorem \ref{thm_distEQ}, can be seen in this context as it indirectly provides a (partial) representation of the adjoint of a Lévy-type operator \((\cA,\domain)\) on the subdomain \(\mathscr{D}(\cA')\cap \cM(\cA,\pi)\). Note that, in our setting, we always assume \(\cE_1=\cE_2=:\cE\)

\begin{corollary}\label{cor_adjoint}
Let \(\cA:\domain\to\cE\) with \(\domain\subset\cE\) dense be a Lévy-type operator with characteristic triplet \((a,b,\Pi)\), and assume \(\test\subset\domain\). Let \(\pi=(\nu,\mu,\rho)\) be a convenient decomposition of the Lévy kernel \(\Pi\). For all measures \(\phi\in\mathscr{D}(\cA')\cap \cM(\cA,\pi)\) on \(\real^d\) it holds
\begin{align}\label{eq_formadj}
(\cA'\phi)\big|_{\test}= -\nabla \cdot ((a+a_\pi) \phi) + \frac12 \nabla \cdot b\nabla\phi + U^\pi(\phi) + \nabla\cdot V^\pi(\phi) + \hess \cdot W^\pi(\phi).
\end{align}
\end{corollary} 

Without further knowledge of the space \(\cE\) and the operator \((\cA,\domain)\), it seems impossible to obtain a general representation of the adjoint of a Lévy-type operator.\\

Note that throughout the literature various choices for \(\cE\) have been used, and ultimately, the decision depends on the desired results and the respective situation. To name a few examples: If the goal is to use Hilbert space techniques, a good option is \(\cE=L^2(\real^d)\), cf. \cite{albeverio2,albeverio}. For a very general viewpoint one may set \(\cE=\cB_b(\real^d)\), cf. \cite{bottcher, sandric}. In the context of Markov processes, in particular Feller processes, one typically either sees this choice, i.e. \(\cE=\cB_b(\real^d)\), or \(\cE=\cC_\infty(\real^d)\), cf. \cite{liggett}. \\

We conclude this section with two simple examples, both of which illustrate that the choice of the underlying Banach spaces may decide upon whether an (infinitesimally) invariant measure of a Lévy-type operator is contained in the domain of the adjoint operator or not. 

\begin{example}
Let \(\mathbf{e}(y):=\ee^{-y}\mathds{1}_{y\geq 0}\) and set $$\ccA f (x) = \int_\real (f(x+y) -f(x)) \mathbf{e}(y)\diff y = f*\mathbf{e}(x) - f(x).$$ Then $\ccA$ is the Feller generator of a compound Poisson process \((X_t)_ {t\geq0}\) with exponentially distributed jumps and  \(\cC_c^\infty(\real)\subset\mathscr{D}(\ccA)\), cf. \cite[Thm. 31.5 and Ex. 31.8]{sato2nd}. 
Clearly, the Lebesgue measure \(\diff x\) is infinitesimally invariant for \((\ccA,\mathscr{D}(\ccA))\) since 
\begin{align*}
\int_\real f*\mathbf{e}(x)\diff x = \int_\real f(x)\diff x \int_\real \mathbf{e}(y)\diff y = \int_\real f(x)\diff x,
\end{align*}
by Fubini and due to the translation invariance of the Lebesgue measure. Further, it is invariant for \((X_t)_{t\geq0}\) as well, cf. \cite[Thm 24.24]{sato2nd}. \\
However, as mentioned above, in the context of Feller processes one often chooses \(\cE=\cB_b(\real)\) or \(\cE=\cC_\infty(\real)\), and as both of their dual spaces are subspaces of the space of bounded measures on \(\real\), in both cases \(\diff x\notin\cE'\).
\end{example}

\begin{example}
Consider the SDE
\begin{align}\label{eq_sdeN}
\diff X_t = \phi  (X_{t-})\diff N_t, \quad X_0 \in\real,
\end{align}
where \((N_t)_{t\geq0}\) is a Poisson process with intensity \(\lambda=1\), and \(\phi \in\cC^\infty(\real)\) such that \(\phi(x)=2\) for all \(x<-\frac12\), \(\phi(x)=-2\) for all \(x>\frac12\), and \(\phi(x)\in[-2,2]\) for all \(x\in[-\frac12,\frac12]\). Then \textbf{(a)} - \textbf{(c)} of Section \ref{secSDEsublinear} are fulfilled, and thus by Theorem \ref{thm_sol} the SDE \eqref{eq_sdeN} possesses a unique solution \(\X\) which is a Feller process and has the pointwise generator \((\cA,\domain)\) given by
\begin{align*}
\cA f(x) = f(x+\phi(x))-f(x)
\end{align*}
for at least \(f\in\testo\). 
In fact, using Itô's formula and proceeding as in (2) of the proof of Theorem \ref{thm_super}, we can show that \(\domain=\cC_\infty(\real)\).  Moreover, with \cite[Lem. 1.31]{bottcher}, we even get \((\ccA,\mathscr{D}(\ccA))=(\cA,\domain)\).\\
It is easy to see that, since \(\X\) is not irreducible, it has multiple different invariant distributions, which might be limiting distributions  depending on \(X_0\). One such example is \(\eta:=\frac12(\delta_{-1}+\delta_{1})\). Indeed, for this choice one easily verifies that
\begin{align*}
\int_\real \ccA f(x) \eta(\diff x)= \frac12\left(f(1)-f(-1)+f(-1)-f(1)\right)=0.
\end{align*}
As \(\eta\in \cC_\infty(\real)'\) the above described approach via the adjoint operator would be feasible here to obtain this (and other) invariant measures.\\ 
However, e.g. considering the Lévy-type operator \((\ccA,\cC_\infty(\real)\cap L^2(\real))\) as operator on \(\cE=~L^2(\real)\), yields that \(\eta\) is infinitesimally invariant for \((\ccA,\cC_\infty(\real)\cap L^2(\real))\), but \(\eta\notin\cE'=L^2(\real)\). To eliminate solutions that are not contained in $\cE'$, the authors of \cite{albeverio}, in which the Hilbert space setting is used, introduce the notion of \textbf{regular} infinitesimally invariant measures, which are assumed to be absolutely continuous with twice integrable densities.
\end{example}

\bibliographystyle{plain}
\bibliography{InvLaw}

\end{document}